\documentclass[a4paper,11pt]{amsart}

\usepackage[utf8]{inputenc}  
\usepackage[T1]{fontenc}  
\usepackage{amscd,amssymb,amsmath,latexsym,graphicx,epsfig,color,url}
\usepackage[active]{srcltx}
\usepackage[all]{xy}
\usepackage{hyperref}
\hypersetup{breaklinks=true}

\def\ov#1{{\overline{#1}}}

\def\wt#1{{\widetilde{#1}}}

\newcommand{\MV}{\operatorname{MV}}
\newcommand{\conv}{\operatorname{conv}}
\newcommand{\newton}{\operatorname{N}}
\newcommand{\vol}{\operatorname{vol}}

\newcommand{\Hom}{\operatorname{Hom}}

\newcommand{\Spec}{\operatorname{Spec}}

\newcommand{\Res}{\operatorname{Res}}
\newcommand{\Elim}{\operatorname{Elim}}

\newcommand{\rank}{\operatorname{rank}}

\renewcommand{\div}{\operatorname{div}}

\newcommand{\e}{{\rm e}}
\renewcommand{\i}{{\rm i}}

\newcommand{\coeff}{\operatorname{coeff}}
\newcommand{\Id}{{\rm id}}
\newcommand{\ord}{\operatorname{ord}}
\newcommand{\norm}{\operatorname{norm}}

\newcommand{\h}{\operatorname{h}}
\newcommand{\Div}{\operatorname{Div}}
\newcommand{\sat}{{\rm sat}}
\newcommand{\pr}{{\rm pr}}

\newcommand{\val}{\operatorname{val}}
\newcommand{\car}{\operatorname{char}}
\newcommand{\bfdeg}{\operatorname{\bf deg}}

\newcommand{\dd}{\, {\rm d}}
\newcommand{\init}{\operatorname{init}}

\newcommand{\A}{\mathbb{A}}
\newcommand{\C}{\mathbb{C}}
\newcommand{\F}{\mathbb{F}}

\newcommand{\K}{\mathbb{K}}
\renewcommand{\L}{\mathbb{L}}

\newcommand{\N}{\mathbb{N}}
\renewcommand{\P}{\mathbb{P}}
\newcommand{\Q}{\mathbb{Q}}
\newcommand{\R}{\mathbb{R}}
\newcommand{\T}{\mathbb{T}}
\newcommand{\Z}{\mathbb{Z}}

\newcommand{\cA}{{\mathcal A}}
\newcommand{\cB}{{\mathcal B}}

\newcommand{\bfa}{{\boldsymbol{a}}}
\newcommand{\bfb}{{\boldsymbol{b}}}
\newcommand{\bfc}{{\boldsymbol{c}}}
\newcommand{\bfd}{{\boldsymbol{d}}}
\newcommand{\bfe}{{\boldsymbol{e}}}
\newcommand{\bff}{{\boldsymbol{f}}}

\newcommand{\bfn}{{\boldsymbol{n}}}

\newcommand{\bft}{{\boldsymbol{t}}}
\newcommand{\bfu}{{\boldsymbol{u}}}

\newcommand{\bfx}{{\boldsymbol{x}}}

\newcommand{\bfz}{{\boldsymbol{z}}}

\newcommand{\bfN}{{\boldsymbol{N}}}
\newcommand{\bfF}{{\boldsymbol{F}}}
\newcommand{\bfcA}{{\boldsymbol{\cA}}}

\newcommand{\bftheta}{{\boldsymbol{\theta}}}
\newcommand{\bfzeta}{{\boldsymbol{\zeta}}}

\newcommand{\bfxi}{{\boldsymbol{\xi}}}

\newcommand{\bfzero}{{\boldsymbol{0}}}
\newcommand{\bfone}{{\boldsymbol{1}}}

\newcounter{thm}
\numberwithin{thm}{section}
\numberwithin{equation}{section}

\theoremstyle{definition}
\newtheorem{definition}[thm]{Definition}

\newtheorem{remark}[thm]{Remark}
\newtheorem{example}[thm]{Example}

\theoremstyle{plain}
\newtheorem{lemma}[thm]{Lemma}
\newtheorem{proposition}[thm]{Proposition}
\newtheorem{theorem}[thm]{Theorem}
\newtheorem{corollary}[thm]{Corollary}

\begin{document}

\title{A Poisson formula for the sparse resultant}

\author[D'Andrea]{Carlos D'Andrea}
\address{D'Andrea:
Departament d'{\`A}lgebra i Geometria, Universitat de Barcelona.
Gran Via~585, 08007 Barcelona, Spain}
\email{cdandrea@ub.edu}
\urladdr{\url{http://atlas.mat.ub.es/personals/dandrea/}}

\author[Sombra]{Mart{\'\i}n~Sombra}
\address{Sombra: ICREA and Departament d'{\`A}lgebra i Geometria, Universitat de Barcelona.
Gran Via~585, 08007 Barcelona, Spain}
\email{sombra@ub.edu}
\urladdr{\url{http://atlas.mat.ub.es/personals/sombra/}}

\date{\today} 
\thanks{D'Andrea was partially
  supported by the MICINN research project MTM2010-20279.  Sombra was
  partially supported by the MINECO research project MTM2012-38122-C03-02.}

\begin{abstract}
  We present a Poisson formula for sparse resultants and a formula for
  the product of the roots of a family of Laurent polynomials, which are
  valid for arbitrary families of supports.

  To obtain these formulae, we show that the sparse resultant
  associated to a family of supports can be identified with the
  resultant of a suitable multiprojective toric cycle in the sense of
  R\'emond.  This connection allows to study sparse resultants using
  multiprojective elimination theory and intersection theory of toric
  varieties.
\end{abstract}
\maketitle

\section{Introduction} \label{sec:introduction}

Sparse resultants are widely used in polynomial
equation solving, a fact that has sparked a lot of interest in their
computational and applied aspects, see for instance
\cite{CanEmi:sbasr, Stu02, DAndrea:msfsr,
  JeronimoKrickSabiaSombra:cccf, CLO05,DicEmi2005:spe,
  JeronimoMateraSolernoWaissbein:dtss}. They have also been studied
from a more theoretical point of view because of their connections
with combinatorics, toric geometry, residue theory, and hypergeometric
functions \cite{GKZ94, Stu94, CatDicStu:rr, Koh97:npnfmvprse,
  CatDicStu:rhf, Est10:npdp}.

{Sparse elimination theory} focuses on ideals and varieties defined by
Laurent polynomials with given supports, in the sense that the
exponents in their monomial expansion are \emph{a priori}
determined. The classical approach to this theory consists in
regarding such Laurent polynomials as global sections of line bundles
on a suitable projective {toric variety}.  Using this interpretation,
sparse elimination theory can be reduced to projective elimination
theory.  In particular, sparse resultants can be studied \emph{via}
the Chow form of this projective toric variety as it is done in
\cite{PS93, GKZ94, Stu94}.  This approach works correctly when all
considered line bundles are very ample, but might fail otherwise. In
particular, important results obtained in this way, like the product
formulae due to Pedersen and Sturmfels \cite[Theorem 1.1 and
Proposition 7.1]{PS93}, do not hold for families of Laurent
polynomials with arbitrary supports.

In this paper, we define and study sparse resultants using the
{multiprojective} elimination theory introduced by R\'emond in
\cite{Remond2001:em} and further developed in our joint paper with
Krick \cite{DKS2011:hvmsan}.  This approach gives a better framework
to understand sparse elimination theory. In particular, it allows to
understand precisely in which situations some classical formulae for
sparse resultants hold, and how to modify them to work in general.

In precise terms, let $M\simeq \Z^{n}$ be a lattice of rank $n\ge0 $
and $N=\Hom(M,\Z)$ its dual lattice. Let $\T_{M}=\Hom(M,\C^{\times})\simeq
(\C^{\times})^{n}$ be the associated algebraic torus over~$\C$ and,
for $a\in M$, we denote by $\chi^{a}\colon \T_{M}\to \C^{\times}$ the
corresponding character.

Let $\cA_{i}$, $i=0,\dots, n$, be a family of $n+1$ nonempty finite
subsets of $M$ and put
\begin{displaymath} 
\bfcA=(\cA_{0}, \dots, \cA_{n}).
\end{displaymath}
For each $i$, consider a set
of $\#\cA_{i}$ variables  $\bfu_{i}=\{u_{i,a}\}_{a\in \cA_{i}}$ and let
\begin{equation}\label{Fi}
  F_{i}=\sum_{a\in \cA_{i}}u_{i,a} \chi^{a}\in \C[\bfu_{i}][M] 
\end{equation}
be the general Laurent polynomial with support $\cA_{i}$, where
we denote by $\C[\bfu_{i}][M]\simeq \C[\bfu_{i}][t_{1}^{\pm1}, \dots,
t_{n}^{\pm1}]$ the group $\C[\bfu_{i}]$-algebra of $M$.
Consider also the incidence variety given by
\begin{displaymath} 
  \Omega_{\bfcA}=\{(\xi,\bfu) \mid
  F_{0}(\bfu_{0},\xi)=\dots=F_{n}(\bfu_{n},\xi)=0\} \subset \T_{M}\times
  \prod_{i=0}^{n}\P(\C^{\cA_{i}}).
\end{displaymath} 
It is a subvariety of codimension $n+1$ defined over $\Q$.

We define the \emph{$\bfcA$-resultant} or \emph{sparse resultant},
denoted by $\Res_{\bfcA}$, as any primitive polynomial in
$\Z[\bfu_{0},\dots, \bfu_{n}]$ giving an equation for the direct image
$\pi_{*}\Omega_{\bfcA}$ (Definition~\ref{def:4}) where 
\begin{displaymath}
\pi\colon \T_{M}\times
\prod_{i=0}^{n}\P(\C^{\cA_{i}})\to \prod_{i=0}^{n}\P(\C^{\cA_{i}})
\end{displaymath}
is the projection onto the second factor. It is well-defined up
to a sign. This notion of sparse resultant coincides with the one
proposed by Esterov in \cite[Definition~3.1]{Est10:npdp}.

The informed reader should be aware that the $\bfcA$-resultant is
usually defined as an irreducible polynomial in $\Z[\bfu]$ giving an
equation for the Zariski closure $\ov{\pi(\Omega_{\bfcA})}$, if
this is a hypersurface, and as $1$ otherwise, as it is done in \cite{GKZ94, Stu94}.  In
this paper, we call this object the \emph{$\bfcA$-eliminant} or
\emph{sparse eliminant} instead, and we denote it by $\Elim_{\bfcA}$.
It follows from these definitions that 
\begin{equation} \label{eq:13}
  \Res_{\bfcA}=\pm\Elim_{\bfcA}^{d_{\bfcA}},
\end{equation}
with $d_{\bfcA}$ equal to the degree of the restriction of $\pi$ to
the incidence variety $\Omega_{\bfcA}$.  This degree is not
necessarily equal to 1 and so, in general, the sparse resultant and the
sparse eliminant are different objects, see Example \ref{exm:5}.

The definition of the sparse resultant in terms of a direct image rather
than just a set-theoretical image, has better properties and produces more uniform
statements. For instance, the partial degrees of the sparse resultant
are given, for $i=0,\dots, n$,  by 
\begin{equation*}
  \deg_{\bfu_{i}}(\Res_{\bfcA})= \MV_{M}(\Delta_{0}, \dots,
  \Delta_{i-1},\Delta_{i+1},\dots, \Delta_{n}),
\end{equation*}
where $\Delta_{i} \subset M_{\R}$ is the lattice polytope given as the
convex hull of $\cA_{i}$ and  $\MV_{M}$ is the mixed volume
function associated to the lattice $M$ (Proposition
\ref{prop:6}). This equality holds for \emph{any} family of supports,
independently of their combinatorics.

One of our motivations  comes from the need of a general Poisson
formula for sparse resultants for our joint work with Galligo on the
distribution of  roots  of families of Laurent polynomials
\cite{DGS2012:qessspe}. By a {\it Poisson formula} 
we mean  an equality of the form 
\begin{equation*}
  \Res_{\bfcA}(f_{0},f_{1},\dots, f_{n})= Q(f_{1},\dots, f_{n}) 
  \cdot \prod_{\xi} f_{0}(\xi)^{m_{\xi}} ,
\end{equation*}
 where $f_{i}\in\C[M]$ is a generic Laurent
polynomial  with support $\cA_{i}$, $i=0,\dots, n$,
 the product is over the roots $\xi$ of $f_{1},\dots, f_{n}$ in 
\begin{math}
  \T_{M},
\end{math}
$m_\xi$ is the multiplicity of $\xi$, and $Q\in \Q(\bfu_{1},\dots,
\bfu_{n})^{\times}$ is a rational function to be determined.  A
formula of this type was stated by Pedersen and Sturmfels in
\cite{PS93} but it does not hold for arbitrary supports. An attempt to
make it valid in full generality was made by Minimair in \cite{Min03},
but his approach has some inaccuracies. 

The main result of this paper is the Poisson formula for the sparse
resultant given below, which holds for {\em any}  family of
supports.  We introduce some notation to state this properly.

Let $v\in N\setminus \{0\}$ and put $v ^{\bot}\cap M \simeq \Z^{n-1}$
for its orthogonal lattice. For $i=1,\dots, n$, we set $\cA_{i,v}$
for the subset of points of $\cA_{i}$ of minimal weight in the
direction of $v$. This gives a family of $n$ nonempty finite subsets
of translates of the lattice $v ^{\bot}\cap M$.  We denote by
$\Res_{\cA_{1,v },\dots,\cA_{n,v}}$ the corresponding  sparse
resultant, also called the sparse resultant of 
$\cA_{1},\dots, \cA_{n}$ in the direction of $v$. Given
Laurent polynomials $f_{i}\in\C[M]$ with support $\cA_{i}$,
$i=1,\dots, n$, we denote by
\begin{displaymath}
  \Res_{\cA_{1,v
  },\dots,\cA_{n,v}}(f_{1,v}, \dots, f_{n,v}) \in \C
\end{displaymath}
the evaluation of this directional sparse resultant at the coefficients of
the initial part of the $f_{i}$'s in the direction of $v$, see Definition \ref{def:5}
for details.
We also set $h_{\cA_{0}}(v)= \min_{a\in \cA_{0}}\langle v,a\rangle$
for the value at $v$ of the support function of~$\cA_{0}$.

\begin{theorem} \label{thm:3} Let $\cA_i\subset M$ be a nonempty
  finite subset and $f_i\in \C[M]$ a Laurent polynomial with support
  contained in $ \cA_i$, $i=0,\dots,n$. Suppose that for all $v \in
  N\setminus \{0\}$ we have that $\Res_{\cA_{1,v },\dots,\cA_{n,v
    }}(f_{1,v },\dots,f_{n,v })\ne 0$. Then
\begin{equation*}
  \Res_{\cA_{0},\cA_1,\ldots,\cA_n}(f_{0},f_{1},\dots,f_{n})= \pm
\prod_{v }\Res_{\cA_{1,v},\dots,\cA_{n,v}}(f_{1,v},\dots,f_{n,v})^{-h_{\cA_{0}}(v )}   \cdot \prod_{\xi} f_{0}(\xi)^{m_{\xi}},
\end{equation*}
the first product being over the primitive vectors $v\in N$ and the
second over the roots $\xi$ of $f_{1},\dots, f_{n}$ in $\T_{M}$, and
where $m_\xi$ denotes the multiplicity of $\xi$.
\end{theorem}

Both products in the above formula are finite.  Indeed,
$\Res_{\cA_{1,v},\dots,\cA_{n,v}}\ne1$ only if $v$ is an inner
normal to a facet of the Minkowski sum
$\sum_{i=1}^{n}\Delta_{i}$. Moreover, by Bernstein theorem
\cite[Theorem B]{Ber75}, the hypothesis that no directional sparse resultant
vanishes implies that the set of roots of the family $f_{i}$,
$i=1,\dots, n$, is finite.

\begin{example} \label{exm:4}
Let $M=\Z^{2}$ and consider the family of nonempty
finite subsets of $\Z^{2}$
\begin{displaymath}
\cA_{0}=\cA_{1}=\{(0,0), (-1,0),(0,-1)\}, \quad       
\cA_{2}=\{(0,0),(1,0), (0,1), (0,2)\}. 
\end{displaymath}
Consider also a family of generic 
Laurent polynomials in $\C[M]=\C[t_{1}^{\pm1},t_{2}^{\pm1}]$
supported in these subsets, that is 
\begin{displaymath}
  f_{i}=\alpha_{i,0}+\alpha_{i,1}\, t_{1}^{-1}+\alpha_{i,2}\,
  t_{2}^{-1}, i=0,1, \quad  f_{2}=\alpha_{2,0}+\alpha_{2,1}\,
  t_{1}+\alpha_{2,2}\, t_{2}
  +\alpha_{2,3}\, t_{2}^{2}.
\end{displaymath}
with $\alpha_{i,j}\in \C$. 

The resultant $\Res_{\cA_{0},\cA_{1}, \cA_{2} }$ is a polynomial in
two sets of 3 variables and a set of~4~variables. It is
multihomogeneous of multidegree $(3,3,1)$ and has 24 terms. 

Considering the Minkowski sum $\Delta_{1}+\Delta_{2}$ we obtain that,
in this case, the only nontrivial directional sparse resultants are
those corresponding to the vectors $(1,0)$, $(1,1)$, $(0,1)$,
$(-1,0)$, $(-2,-1)$, and $(0,-1)$. Computing them together with their
corresponding exponents in the Poisson formula, Theorem \ref{thm:3}
shows that
\begin{displaymath}
  \Res_{\cA_{0},\cA_{1}, \cA_{2} }(f_{0},f_{1},f_{2})= \pm
  \alpha_{1,2}\, \alpha_{1,1}^{2} \, \alpha_{2,0}\prod_{i=1}^{3}f_{0}(\xi_{i}).
\end{displaymath}
where the $\xi_{i}$'s are the solutions of the system of equations
$f_{1}=f_{2}=0$. 
    
Each of the supports generates the lattice $\Z^{2}$, and so
$\Elim_{\bfcA}=\Res_{\bfcA}$. However, the formula in \cite[Theorem
1.1]{PS93} gives in this case an exponent 1 to the coefficient
$\alpha_{1,1}$, instead of 2.  Hence, this formula does not work in
this case.  Minimair's reformulation of the Pedersen-Sturmfels formula
in \cite[Theorem 8]{Min03} gives an expression for the exponent of
$\alpha_{1,1}$ that evaluates to $\frac00$, and so it also fails in
this case.
\end{example}

As a by-product of our approach, we obtain a formula for the product
of the roots of a family of Laurent polynomials. For a nonzero complex
number $\gamma\in \C^{\times}$ and $v\in N$, we consider the point in
the torus $\gamma^{v}\in \T_{M}$ given by the homomorphism $M\to
\C^{\times}$, $a\mapsto \langle a,v\rangle$. 

    \begin{corollary} \label{cor:4} Let $\cA_i\subset M$ be a nonempty
      finite subset and $f_i\in \C[M]$ a Laurent polynomial with
      support contained in $ \cA_i$, $i=1,\dots,n$. Suppose that for all $v \in N\setminus
      \{0\}$ we have that
      $\Res_{\cA_{1,v },\dots,\cA_{n,v }}(f_{1,v
      },\dots,f_{n,v })\ne 0$. Then
  \begin{displaymath}
\prod_{\xi} \xi^{m_{\xi}}= \pm
  \prod_{v}\Res_{\cA_{1, v},\dots,\cA_{n,v}} (f_{1,v},\dots, f_{n,v}) ^v,
  \end{displaymath}
  the first product being over the roots $\xi$ of $f_{1},\dots, f_{n}$
  in $\T_{M}$ and the second over the primitive vectors $v\in N$, and
  where $m_\xi$ denotes the multiplicity of $\xi$.
Equivalently, for $a\in M,$
  \begin{displaymath}
\prod_{\xi} \chi^{a}(\xi)^{m_{\xi}}= \pm
  \prod_{v}\Res_{\cA_{1,v},\dots,\cA_{n,v}} (f_{1,v},\dots, f_{n,v}) ^{\langle
    a,v\rangle}.
  \end{displaymath}
\end{corollary}

This result makes explicit both the scalar factor and the exponents
in Khovanskii's formula in \cite[\S 6, Theorem 1]{Koh97:npnfmvprse}.

As a consequence of the Poisson formula in Theorem \ref{thm:3}, we
obtain an extension to the sparse setting of the ``hidden variable''
technique for solving polynomial equations, which is crucial for
computational purposes \cite[\S3.5]{CLO05}, see Theorem \ref{thm:4}
below. 

To do this, let $n\ge 1$ and set $M=\Z^{n}$ and, for $i=1,\dots, n$,
consider the general Laurent polynomials $F_{i}\in
\Z[\bfu_{i}][t_{1}^{\pm1}, \dots, t_{n}^{\pm1}]$ with support
$\cA_{i}$ as in \eqref{Fi}. Each $F_{i}$ can be alternatively
considered as a Laurent polynomial in the variables $\bft' :=
\{t_1,\dots,t_{n-1}\}$ and coefficients in the ring
$\Z[\bfu_{i}][{t_n}^{\pm1}]$. In this case, we denote it by
$F_{i}(\bft')$.  The support of this Laurent polynomial is the
nonempty finite subset $\varpi(\cA_i)\subset \Z^{n-1},$ where $\varpi
\colon \R^{n}\rightarrow \R^{n-1}$ denotes the projection onto the
first $n-1$ coordinates of $\R^{n}$. We then set
\begin{equation}
  \label{eq:42}
 \Res_{\cA_{1}, \dots, \cA_{n}}^{t_{n}}=\Res_{\varpi(\cA_1),\dots,\varpi(\cA_{n})}
(F_1(\bft') ,\dots, F_n(\bft')) \in \C[\bfu_{1},\dots, \bfu_{n}][t_n^{\pm1}].  
\end{equation}

In other words, we ``hide'' the variable $t_{n}$ among the
coefficients of the $F_{i}$'s and we consider the corresponding sparse
resultant.

The following result shows that the roots of this Laurent polynomial
coincide with the $t_{n}$-coordinate of the roots of the family
$f_{i}$, $i=1,\dots, n$, and that their corresponding multiplicities are preserved. 
It generalizes and precises \cite[Proposition 5.15]{CLO05},
which is stated for generic families of dense polynomial equations.

\begin{theorem} \label{thm:4}
  Let $\cA_i\subset \Z^{n}$ be a nonempty finite subset and $f_i\in
  \C[t_{1}^{\pm1},\dots, t_{n}^{\pm1}]$ a Laurent polynomial with
  support contained in $ \cA_i$, $i=1,\dots,n$.  Suppose that for all
  $v \in \Z^{n}\setminus \{0\}$ we have that $\Res_{\cA_{1,v
    },\dots,\cA_{n,v }}(f_{1,v },\dots,f_{n,v })\ne 0$. Then
  there exist $\lambda\in \C^{\times}$ and $d\in \Z$ such that
  \begin{equation}    \label{eq:79}
\Res_{\cA_{1}, \dots, \cA_{n}}^{t_{n}}(f_{1},\dots,
  f_{n})=\lambda \,  t_{n}^{d}\, \prod_{\xi}(t_{n}-\xi_{n})^{m_{\xi}}, 
  \end{equation}
the product being being over the roots
  $\xi=(\xi_{1},\dots, \xi_{n})$ of $f_{1},\dots, f_{n}$ in
  $(\C^{\times})^{n}$, and where $m_\xi$ denotes the multiplicity of
  $\xi$.
\end{theorem}

Indeed, the exponent $d$ in \eqref{eq:79} can be made explicit in
terms of ``mixed integrals'' in the sense of \cite[Definition
1.1]{PS:rbke} or, equivalently, ``shadow mixed volumes'' as in
\cite[Definition 1.7]{Esterov:emfb}, see Remark \ref{rem:2} for
further details. 

In addition, we also obtain a product formula for the addition of supports
(Corollary~\ref{cor:3}) and we extend the height bound for the sparse
resultant in \cite[Theorem~1.1]{Som04} to arbitrary collections of
supports (Proposition \ref{prop:4}).

The exponent $d_{\bfcA}$ in \eqref{eq:13} can be expressed in
combinatorial terms (Proposition~\ref{prop:10}). Hence, all formulae
and properties for the sparse resultant can be restated for sparse
eliminants at the cost of paying attention to the relative position of
the supports with respect to the lattice generated by an essential
subfamily, see \S \ref{sec:basic-prop-sparse} for details.

Our approach is based on multiprojective elimination theory. Let
$Z_{\bfcA}$ be the multiprojective toric cycle associated to the
family $\bfcA,$ and denote by $|Z_{\bfcA}|$ its supporting
subvariety. In Proposition~\ref{prop:1} we show that
\begin{displaymath}
  \Elim_{\bfcA}=\pm \Elim_{\bfe_{0}, \dots, \bfe_{n}}(|Z_{\bfcA}|), \quad   \Res_{\bfcA}=\pm \Res_{\bfe_{0}, \dots, \bfe_{n}}(Z_{\bfcA}),
\end{displaymath}
were $\Elim_{\bfe_{0}, \dots, \bfe_{n}}$ and $\Res_{\bfe_{0}, \dots,
  \bfe_{n}}$ respectively denote the eliminant and the resultant
associated to the vectors $\bfe_{i}$, $i=0,\dots, n$, in the standard
basis of $\Z^{n+1}$, see \S \ref{sec:result-elim-mult} for
details. Both eliminants and resultants play an important role in this
theory, but it is well known that multiprojective resultants are {the} central
objects because they reflect better the geometric operations at an
algebraic level. 

Our proof of Theorem \ref{thm:3} is based on the standard properties
of multiprojective resultants and on tools from toric geometry,
together with the classical Bernstein's theorem and its refinement for
valued fields due to Smirnov \cite{Smi97:tsodvr}.

We remark that the formula in \cite[Theorem 1.1]{PS93} is stated for
general Laurent polynomials and that it amounts to an equality modulo
an unspecified scalar factor in $\Q^{\times}$.  In Theorem
\ref{thm:1}, we extend this product formula to an arbitrary family of
supports and we precise the value of this scalar factor up to a sign.
Theorem \ref{thm:3} follows from this result after showing that the
formula in Theorem~\ref{thm:1} can be evaluated into a particular
family of Laurent polynomials exactly when no directional sparse
resultant vanishes.

The paper is organized as follows: in \S \ref{sec:preliminaries} we
introduce some notation and show a number of preliminary results
concerning intersection theory on multiprojective spaces, toric
varieties and cycles, and root counting on algebraic tori.  In
\S \ref{sec:basic-prop-sparse} we show that the sparse eliminant and
the sparse resultant respectively coincide with the eliminant and the
resultant of a multiprojective toric cycle and, using this
interpretation, we derive some of their basic properties from the
corresponding ones for general eliminants and resultants.  In
\S \ref{sec:poiss-form-form} we prove the Poisson formula for the
sparse resultant and we derive  some of its consequences.  In
\S \ref{sec:example-comp-with} we give some more examples, compare our
results with previous ones, and establish sufficient
conditions for these previous results to hold.

\medskip
\noindent {\bf Acknowledgments.}
We thank Jos\'e Ignacio Burgos, Alicia Dickenstein, Alex
Esterov, Juan Carlos Naranjo and Bernd Sturmfels for
several useful discussions and suggestions.

\section{Preliminaries}\label{sec:preliminaries}

All along this text, bold symbols indicate finite sets or sequences of objects, where the
type and number should be clear from the context. For instance, $\bfx$
might denote the set of variables $\{x_{1},\dots, x_{n}\}$ so that,
if $K$ is a field, then $K[\bfx]=K[x_{1},\dots,x_{n}]$.  

We denote by $\N$ the set of nonnegative integers. Given a vector
$\bfb=(b_{1},\dots, b_{n})\in \N^{n}$, we set
$|\bfb|=\sum_{i=1}^{n}b_{i}$ for its length.

\subsection{Cycles on multiprojective
  spaces}\label{sec:cycl-mult-space}

In this subsection, we give the notation and basic facts on
intersection theory of multiprojective spaces. Most of the material is
taken from \cite[\S 1.1 and 1.2]{DKS2011:hvmsan}.

Let $K$ be a field and $\K$ an algebraically closed field containing
$K$. For instance, $K$ and $\K$ might be taken as $\Q$ and $\C$,
respectively. For $m\ge 0$ and $\bfn=(n_{0},\dots,n_{m})\in
\N^{m+1}$, we consider the multiprojective space over~$K$ given by
\begin{equation}\label{eq:21}
\P^{\bfn}_{K}=\P^{n_{0}}_{K}\times\dots\times\P^{n_{m}}_{K}.
\end{equation}
For $i=0,\dots, m$, let $\bfx_{i}=\{x_{i,0},\dots, x_{i,n_{i}}\}$ be a
set of $n_{i}+1$ variables and put $\bfx=\{\bfx_{0},\dots,
\bfx_{m}\}$. The multihomogeneous coordinate ring of $\P_{K}^{\bfn}$
is then given by $ K[\bfx]= K[\bfx_0, \dots, \bfx_m]$. It is
multigraded by declaring $\bfdeg(x_{i,j}) = \bfe_i\in \N^{m+1}$, the
$(i+1)$-th vector of the standard basis of $\R^{m+1}$.  For
$\bfd=(d_0,\dots,d_m)\in \N^{m+1}$, we denote by $ K[\bfx]_{\bfd}$ its
component of multidegree $\bfd$.

Set
\begin{equation}\label{eq:8}
\N^{n_{i}+1}_{d_{i}}=\{\bfa_{i}\in \N^{n_{i}+1}\mid
|\bfa_{i}|=d_{i}\}\quad \text{ and } \quad
\N^{\bfn+\bfone}_{\bfd}= \prod_{i=0}^{m} \N^{n_i+1}_{d_{i}}.
\end{equation}
With this notation, a multihomogeneous polynomial $f\in K[\bfx]_{\bfd}$ writes
down as
\begin{displaymath}
f =  \sum_{\bfa\in \N^{\bfn+\bfone}_{\bfd}} \alpha_{\bfa} \,
\bfx^\bfa
\end{displaymath}
where, for each index $\bfa\in \N^{\bfn+\bfone}_{\bfd}$,
$\alpha_{\bfa}$ denotes an element of $K$ and
$\bfx^\bfa=\prod_{i,j}x_{i,j}^{a_{i,j}}$.

A \emph{cycle} on $\P^{\bfn}_{K}$ is a  $\Z$-linear combination 
\begin{equation} \label{eq:51}
  X=\sum_{V}m_{V} V
\end{equation}
where the sum is over the irreducible subvarieties $V$ of
$\P^{\bfn}_{K}$ and $m_{V}=0$ for all but a finite number of $V$. The
subvarieties $V$ such that $m_{V}\ne0$ are called the
\emph{irreducible components} of $X$. The \emph{support} of $X$,
denoted by $|X|$, is the union of its irreducible components.
We also denote by $X_{\K }$ the cycle on $\P^{\bfn}_{\K }$ obtained
from $X$ by the base change $ K\hookrightarrow \K$, that is
\begin{displaymath} 
  X_{\K }=\sum_{V}m_{V} (V\times_{\Spec(K)} \Spec(\K )).
\end{displaymath}

A cycle is \emph{equidimensional} or of \emph{pure dimension} if all
its irreducible components are of the same dimension. For $r=0,\dots,
|\bfn|$, we denote by $Z_r(\P^\bfn_K)$ the group of cycles
on~$\P^{\bfn}$ of pure dimension~$r$.

Given a multihomogeneous ideal $I\subset K[\bfx]$, we denote by $V(I)$
the subvariety of $\P^{\bfn}_{K}$ defined by $I$. For each
minimal prime ideal $P$ of $I$, we denote by $m_{P}$ its
\emph{multiplicity}, defined as the length of the $K[\bfx]$-module
$(K[\bfx]/I)_{P}$. We then set
\begin{equation*}
Z(I)=\sum_{P}m_{P}V(P)  
\end{equation*}
for the cycle on~$\P_{K}^{\bfn}$ defined by $I$. 

We denote by $\Div(\P_{K}^{\bfn})$ the group of Cartier divisors on
$\P_{K}^{\bfn}$.  Given a multihomogeneous rational function
$f\in K(\bfx)^{\times}$, we denote by
\begin{equation*} 
\div(f)\in \Div(\P_{K}^{\bfn})  
\end{equation*} 
the associated Cartier divisor. Using \cite[Propositions II.6.2 and
II.6.11]{Hartshorne:ag} and the fact that the ring $K[\bfx]$ is
factorial, we can verify that that every Cartier divisor
on~$\P_{K}^{\bfn}$ is of this form.

Let $X$ be a cycle of pure dimension $r$ and $D\in
\Div(\P_{K}^{\bfn})$ a Cartier divisor intersecting $X$ properly.  We
denote by $X\cdot D$ the intersection product of $X$ and $D$, with
intersection multiplicities as in \cite[\S I.1.7,
page~53]{Hartshorne:ag}, see also \cite[Definition
1.3]{DKS2011:hvmsan}. It is a cycle of pure dimension $r-1$.

Let $X\in Z_{r}(\P^{\bfn}_{K})$ and $\bfb \in \N^{m}_{r}$ a vector of
length $r$. For $i=0,\dots, m$, we denote by
$H_{i,j}\in\Div\big(\P^{\bfn}_{\K }\big)$, $j=1,\dots, b_{i}$, the
inverse image under the projection $ \P^\bfn_{\K }\to \P^{n_i}_{\K}$
of a family of $b_{i}$ generic hyperplanes of $\P^{n_i}_{\K}$. The
\emph{degree} of $X$ of index $\bfb$, denoted by $\deg_{\bfb}(X)$, is
defined as the degree of the 0-dimensional cycle
$$
X_{\K }\cdot \prod_{i=0}^{m}\prod_{j=1}^{b_{i}} H_{i,j}.
$$

The Chow ring of $\P_{K}^{\bfn}$, denoted by $A^{*}(\P_{K}^{\bfn}),$ can be written down  as
\begin{equation}\label{eq:54}
  A^{*}(\P_{K}^{\bfn})=\Z[\theta_{0}, \dots,
  \theta_{m}]/(\theta_{0}^{n_{0}+1},\dots, \theta_{m}^{n_{m}+1})
\end{equation}
where $\theta_{i}$ denotes the class of the inverse image under the
projection $\P^{\bfn}_K\to \P^{n_{i}}_K$ of a hyperplane of
$\P^{n_{i}}_K$ \cite[Example~8.4.2]{Fulton:IT}.  

Given $X\in
Z_r(\P^\bfn_K)$, its class in the Chow ring is
\begin{equation} \label{eq:53}
[X]= \sum_{\bfb} \deg_{\bfb}(X)\, \theta_0^{n_0 -b_0}\cdots
 \theta_m^{n_m -b_m},
\end{equation}
the sum being over all $\bfb\in \N_r^{m+1}$ such that $b_{i}\le n_{i}$
for all $i$. It is a homogeneous element of $A^{*}(\P^{\bfn}_K)$ of
degree $|\bfn|-r$ containing the information of all the mixed
degrees of~$X$.
In the particular case when $X=Z(f)$ with $f\in K[\bfx]_{\bfd}$, we
have that
\begin{equation*} 
  [Z(f)]=\sum_{i=0}^{m}d_{i}\theta_{i}.
\end{equation*}

Let $X\in Z_{r}(\P_{K}^{\bfn})$ and $f\in K[\bfx]$ a multihomogeneous
polynomial such that $X$ and $\div(f)$ intersect properly.  The
multiprojective B\'ezout theorem says that
\begin{equation}  \label{eq:50}
  [X\cdot \div(f)]= [X]\cdot[Z(f)],
\end{equation}
see for instance \cite[Theorem~1.11]{DKS2011:hvmsan}.

\begin{definition}
  \label{def:4}
Let   $\varphi\colon \P^{\bfn_{1}}_{K}\to \P^{\bfn_{2}}_{K}$ be a
morphism and $V$ an
irreducible subvariety of $\P^{\bfn_{1}}_{K}$ of dimension $r$. The \emph{degree} of
$\varphi$ on $V$ is defined as
\begin{displaymath}
  \deg(\varphi|_{V})=\left\{
    \begin{array}{cl}
      [K(V):K({{\varphi(V)}})]& \text{ if }
      \dim({\varphi(V)})= r, \\[2mm]
      0 & \text{ if }
      \dim({\varphi(V)})< r.
    \end{array}\right.
\end{displaymath}
The \emph{direct image} under $\varphi$ of $V$ is defined as
$\varphi_{*}V= \deg(\varphi|_{V}) \,{\varphi(V)}$. It is a cycle of
dimension $r$. This notion extends by linearity to equidimensional
cycles and induces a linear map
\begin{displaymath}
  \varphi_{*}\colon Z_{r}(\P_{K}^{\bfn_{1}})\longrightarrow Z_{r}(\P_{K}^{\bfn_{2}}).
\end{displaymath}


Let $H$ be a hypersuface of $\P^{\bfn_{2}}_{K}$ not containing the
image of $\varphi$. The \emph{inverse image} of~$H$ under $\varphi$ is
defined as the hypersurface $\varphi^{*}H= {\varphi^{-1}(H)}$. This
notion extends by linearity to a $\Z$-linear map
\begin{displaymath}
\varphi^{*}:\Div(\P^{\bfn_{2}}_{K})\dashrightarrow \Div(\P^{\bfn_{1}}_{K}),
\end{displaymath}
well-defined for Cartier divisors whose support does not contain the
image of $\varphi$.
\end{definition}

Direct images of cycles, inverse images of Cartier divisors and
intersection products are related by the \emph{projection formula}
\cite[Chapter V, \S C.7, formula (11)]{Serre:alm}: let
$\varphi:\P^{\bfn_{1}}_{K}\to\P^{\bfn_{2}}_{K}$ be a morphism, $X$ an
equidimensional cycle on $\P^{\bfn_{1}}_{K}$ and $D$ a Cartier divisor
on $\P^{\bfn_{2}}_{K}$ intersecting $\varphi_{*}X$ properly. Then
\begin{equation}
  \label{eq:38}
\varphi_{*}X\cdot D =  \varphi_{*}(X\cdot \varphi^{*}D).
\end{equation}

\begin{lemma} \label{lemm:2}
Let $0\le q\le m$ and denote by
  \begin{math}
   \pr\colon \P_{K}^{\bfn}\rightarrow
  \prod_{i=0}^{q}\P_{K}^{n_{i}} 
\end{math} the projection onto the first $q+1$ factors of
$\P_{K}^{\bfn}$. Let $X\in Z_{r}(\P_{K})$ and $\bfb\in
\N^{q+1}_{r}$. Then
  \begin{displaymath}
    \deg_{\bfb}(\pr_{*}X)=     \deg_{\bfb,\bfzero}(X). 
  \end{displaymath}
\end{lemma}

\begin{proof}
  We suppose without loss of generality that $K$ is algebraically
  closed. We proceed by induction on the dimension of $X$.  For $r=0$
  we have that $X=\sum_{\bfxi} m_{\bfxi}\bfxi$ with $\bfxi\in
  \P_{K}^{\bfn}$ and $m_{\bfxi}\in \Z$, and $\bfb=\bfzero\in
  \N^{q+1}$. Then $\pr_{*}X=\sum_{\bfxi} m_{\bfxi}\, \pr_{*}\bfxi$ and
  so
\begin{displaymath}
  \deg_{\bfzero}(\pr_{*}X)= \sum_{\bfxi}
  m_{\bfxi}= \deg_{\bfzero,\bfzero}(X).
\end{displaymath}

Now let $r\ge 1$. Choose $0\le i_{0}\le q$ such that $b_{i_{0}}\ge 1$
and let $H\in\Div\big( \prod_{i=0}^{q}\P_{K}^{n_{i}}\big) $ be the
inverse image of a generic hyperplane of $\P_{K}^{n_{i_{0}}}$ under
the projection of $\prod_{i=0}^{q}\P_{K}^{n_{i}}$ onto the $i_{0}$-th
factor. This Cartier divisor intersects $\pr_{*}X$ properly and,
by~\eqref{eq:38},
\begin{displaymath}
  \pr_{*}X \cdot H = \pr_{*}(X\cdot \pr^{*}H).
\end{displaymath}
Using this, together with the multiprojective B\'ezout theorem in
\eqref{eq:50} and the inductive hypothesis, we deduce that
\begin{multline*}
  \deg_{\bfb}(\pr_{*}(X))= \deg_{\bfb-\bfe_{i_{0}}}(\pr_{*}X
  \cdot H )= \deg_{\bfb-\bfe_{i_{0}}}(\pr_{*}(X\cdot \pr^{*}H))
  \\ = \deg_{\bfzero,\bfb-\bfe_{i_{0}}}(X\cdot \pr^{*}H)=
  \deg_{\bfzero, \bfb}(X),
\end{multline*}
which proves the statement. 
\end{proof}

We refer to \cite[\S1.2]{DKS2011:hvmsan} for other properties of mixed
degrees of cycles, including their behavior with respect to linear
projections, products and ruled joins.

\subsection{Eliminants and resultants of multiprojective
  cycles} \label{sec:result-elim-mult} In this subsection, we recall
the notions and basic properties of eliminants of varieties and
resultants of cycles following R\'emond \cite{Remond2001:em} and our
joint paper with Krick \cite{DKS2011:hvmsan}. We also give an
alternative definition of these objects with a more geometric flavor,
and show that both coincide
(Proposition~\ref{prop:2}).

Let $A$ be a factorial ring with field of fractions $K$. Let
$\bfn=(n_{0},\dots,n_{m})\in \N^{m+1}$ and let $\P_{K}^{\bfn}$ be the
corresponding multiprojective space as in~\eqref{eq:21}. Given $r\ge
0$ and a family of vectors $\bfd=(\bfd_0, \dots, \bfd_r )\in
(\N^{m+1}\setminus \{\bfzero\})^{r+1}$, we set
\begin{displaymath}
   N_{i}=\# \N_{d_{i}}^{n+1}- 1= \prod_{j=0}^{m}
  {d_{i,j}+n_{j}\choose n_{j}}-1, \quad i=0,\dots, r,
\end{displaymath} 
with $\N_{d_{i}}^{n+1}$ as in \eqref{eq:8}.  We will work in the
multiprojective space $\P^\bfN_{K}= \prod_{i=0}^{r}\P_{K}^{N_{i}}$
with $\bfN=(N_{0},\dots, N_{r})\in \N^{r+1}$.  For each $i$ we
consider a set of $N_{i}+1$ variables
$\bfu_{i}=\big\{u_{i,\bfa}\}_{\bfa\in \N^{\bfn+\bfone}_{\bfd}}$.  The
coordinates of $\P_{K}^{N_{i}}$ are indexed by the elements of
$\N^{n+1}_{d_{i}}$, and so $K[\bfu_{i}]$ is the homogeneous coordinate
ring of $\P_{K}^{N_{i}}$. Hence, if we set $\bfu=\{\bfu_{0},\dots,
\bfu_{r}\}$, then $K[\bfu]$ is the multihomogeneous coordinate ring of
$\P^\bfN_{K}$.

Consider the {general multihomogeneous polynomial} of multidegree $\bfd_i$ given by
\begin{equation} \label{eq:12}
F_i = \sum_{\bfa\in \N^{\bfn+\bfone}_{\bfd_{i}}} u_{i,\bfa} \,
\bfx^\bfa \in K[\bfu_{i}][\bfx],
\end{equation}
and denote by $\div(F_{i})$ the Cartier divisor on
$\P^{\bfn}_{K}\times\P_{K}^{\bfN}$ it defines.  Given $X\in
Z_{r}(\P^{\bfn}_{K})$, the family of Cartier divisors $\div(F_{i})$,
$i=0,\dots, r$, intersects $X \times \P_{K}^{\bfN}$ properly. We then
set
\begin{equation}\label{eq:30}
  \Omega_{X,\bfd}=(X\times \P_{K}^{\bfN}) \cdot \prod_{i=0}^{r}\div(F_{i}),
\end{equation}
which is a cycle on $\P^{\bfn}_{K}\times \P^{\bfN}_{K}$ of pure
codimension $|\bfn|+1$.  When $X=V$ is an irreducible subvariety, it
coincides with the incidence variety of $V$ and~$F_{i}$'s. Consider
also the morphism given by the projection onto the second factor
\begin{equation}\label{eq:90}
\rho\colon \P^{\bfn}_{K}\times \P^{\bfN}_{K} \longrightarrow
\P^{\bfN}_{K}  .
\end{equation}

\begin{definition} \label{def:3} Let $V\subset\P_{K}^{\bfn}$ be an
  irreducible subvariety of dimension $r$ and $\bfd\in
  (\N^{m+1}\setminus \{\bfzero\})^{r+1}$.  The \emph{eliminant} of $V$
  of index $\bfd$, denoted by $\Elim_{\bfd}(V)$, is defined as any
  irreducible polynomial in $A[\bfu]$ giving an equation for the image
  $\rho(\Omega_{V,\bfd})$ if it is a hypersurface, and as 1 otherwise.
\end{definition}

\begin{definition} \label{def:6} Let $V\subset\P_{K}^{\bfn}$ be an
  irreducible subvariety of dimension $r$ and $\bfd\in
  (\N^{m+1}\setminus \{\bfzero\})^{r+1}$. The \emph{resultant} of $V$ of
  index $\bfd$, denoted by 
  $\Res_{\bfd}(X)$, is defined as any primitive polynomial in $A[\bfu]$
  giving an equation for the direct image $\rho_{*}\Omega_{V,\bfd}$.

  More generally, let $X\in Z_{r}(\P_{K}^{\bfn})$ and write $X=\sum_{V}m_{V}V$ as in
  \eqref{eq:51}. Then, the \emph{resultant} of $X$ of index $\bfd$ is
  defined as
  \begin{displaymath}
    \Res_{\bfd}(X)=\prod_{V}\Res_{\bfd}(V)^{m_{V}}.
  \end{displaymath}
\end{definition}

Both eliminants and resultants are well-defined up to an scalar factor in $A^{\times}$,
the group of units of $A$.

The eliminant $\Elim_{\bfd}(V)$ can be alternatively defined as an
irreducible equation for the support of the direct image
$\rho_{*}\Omega_{V,\bfd}$. Hence
\begin{displaymath}
  \Res_{\bfd}(V)= \lambda \Elim_{\bfd}(V)^{\deg(\rho|_{\Omega_{V,\bfd}})}
\end{displaymath}
with $\lambda\in A^{\times}$. The exponent $
\deg(\rho|_{\Omega_{V,\bfd}})$ is not necessarily equal to 1 and so
eliminants and resultants do not necessarily coincide, see for
instance \cite[Example~1.31]{DKS2011:hvmsan}.

The definitions of these objects in
\cite{Remond2001:em, DKS2011:hvmsan} are given in more algebraic
terms. We now show that our present definitions coincide with theirs.

\begin{proposition} \label{prop:2} The notions of eliminants and
  resultants in Definitions \ref{def:3} and \ref{def:6} respectively
  coincide, up to a scalar factor in $A^{\times}$, with those in
  \cite[Definitions~1.25 and 1.26]{DKS2011:hvmsan}.
\end{proposition}

\begin{proof} Let notation be as Definitions \ref{def:3} and
  \ref{def:6}, and denote temporarily by $\wt \Elim_{\bfd}(V)$ and $\wt
  \Res_{\bfd}(V)$ the eliminant and the resultant from
  \cite{DKS2011:hvmsan}.  By Proposition~1.37(2) and
  Lemma 1.34 in {\it loc. cit.}, all four zero sets of $\Elim_{\bfd}(V)$,
  $\Res_{\bfd}(V)$, $\wt \Elim_{\bfd}(V)$ and $\wt \Res_{\bfd}(V)$
  coincide. By construction, both $\Elim_{\bfd}(V)$ and $\wt
  \Elim_{\bfd}(V)$ are irreducible and so they coincide up to a scalar
  factor in $A^{\times}$, proving the statement for the eliminants.

  Both resultants are powers of the same irreducible
  polynomial. Hence, to prove the rest of the statement it is enough
  to show that their mixed degrees coincide.

  Let $0\le i\le r$. By \cite[Propositions 1.10(4)]{DKS2011:hvmsan}
  and Lemma \ref{lemm:2},
\begin{equation}\label{eq:20}
 \deg_{\bfu_{i}}(\Res_{\bfd}(V))=
  \deg_{\bfN-\bfe_{i}}(\rho_{*}\Omega_{V,\bfd}) =   \deg_{\bfzero,\bfN-\bfe_{i}}(\Omega_{V,\bfd}),
\end{equation}
where $\bfe_{i}$ denotes the $(i+1)$-th vector in the standard basis
of $\Z^{r+1}$. 

Let $\theta_{i}$, $i=0,\dots,m$, and $\zeta_{j}$, $j=0,\dots, r$,
respectively denote the variables in the Chow rings
$A^{*}(\P^{\bfn}_{K})$ and $A^{*}(\P_{K}^{\bfN})$ as in \eqref{eq:54}.
Let $[\Omega_{V,\bfd}]$ denote the class of the incidence variety
in the Chow ring $A^{*}(\P^{\bfn}_{K} \times \P^{\bfN}_{K})\simeq
A^{*}(\P^{\bfn}_{K}) \otimes A^{*}(\P^{\bfN}_{K})$. By \eqref{eq:53},
\begin{equation}\label{eq:22}
\deg_{\bfzero,\bfN-\bfe_{i}}(\Omega_{V,\bfd}) = \coeff_{\bftheta^{\bfn} \zeta_{i}} ([\Omega_{V,\bfd}]). 
\end{equation}
By the
multiprojective B\'ezout theorem in~\eqref{eq:50}, 
\begin{math}
  [\Omega_{V,\bfd}]= [V\times\P^{\bfN}] \cdot 
\prod_{i=0}^{n} [Z(F_{i})],
\end{math}
where $F_{i}$ is the general polynomial as in \eqref{eq:12}. 
By \cite[Propositions 1.19(2) and 1.10(2,4)]{DKS2011:hvmsan}, the
classes in $ A^{*}(\P^{\bfn}_{K})\otimes
  A^{*}(\P^{\bfN}_{K})$ of $V\times \P_{K}^{\bfN}$ and $Z(F_{i})$ are given by 
\begin{displaymath}
  [V\times\P^{\bfN}]= [V]\otimes 1 \quad \text{ and } \quad  [Z(F_{i})]=\zeta_{i}
  +\sum_{j=0}^{m}d_{i,j} 
\theta_{j},
\end{displaymath}
where $[V]$ denotes the class of $V$ in $A^{*}(\P^{\bfn}_{K})$. Hence,
\begin{multline} \label{eq:55} \coeff_{\bftheta^{\bfn} \zeta_{i}}
  ([\Omega_{V,d}])= \coeff_{\bftheta^{\bfn}\zeta_{i}} \Big((
  [V]\otimes 1 ) \cdot \prod_{i=0}^{n} \Big(\zeta_{i}
  +\sum_{j=0}^{m}d_{i,j} \theta_{j}\Big)\Big) \\ =
  \coeff_{\bftheta^{\bfn}} \Big([V] \cdot \prod_{\ell\ne i}^{n}
  \sum_{j=0}^{m}d_{\ell,j} \theta_{j}\Big).
\end{multline}
Then, \eqref{eq:20}, \eqref{eq:22} and \eqref{eq:55} together with Proposition 1.32
in {\it loc. cit.}, imply that
\begin{displaymath}
 \deg_{\bfu_{i}}(\Res_{\bfd}(V))=
\deg_{\bfu_{i}}(\wt \Res_{\bfd}(V)).
\end{displaymath}
Hence, both resultants coincide up to a scalar factor in $A^{\times}$.
The general case when $X$ is a cycle of pure dimension $r$ follows by
linearity.
\end{proof}

Let $V\subset\P_{K}^{\bfn}$ be an irreducible subvariety of dimension
$r$ and $\bfd\in (\N^{m+1}\setminus \{\bfzero\})^{r+1}$.  Each set of
variables $\bfu_{i}$ corresponds to the coefficients of a
multihomogeneous polynomial of degree $\bfd_{i}$.  Hence, given
$f_{i}\in \K [\bfx]_{\bfd_{i}}$, $i=0,\dots, r$, we can write
\begin{displaymath}
\Elim_{\bfd}(V)(f_{0},\dots, f_{r}) \quad \text{ and } \quad \Res_{\bfd}(X)(f_{0},\dots, f_{r})
\end{displaymath}
for the evaluation of the eliminant and of the resultant at the
coefficients of the $f_{i}$'s, respectively.

Eliminants and resultants are
polynomials whose vanishing at a given family of multihomogeneous
polynomials corresponds to the condition that this family has a common
root on $V$: if $\rho(\Omega_{V,\bfd})$ is a
hypersurface, then
\begin{equation}\label{eq:56}
\Res_{\bfd}(V)(f_{0},\dots, f_{r})=0 \Longleftrightarrow  V\cap
V(f_{0},\dots, f_{r})\ne \emptyset,  
\end{equation}
and a similar statement holds for the eliminant. 

A central property of resultants is that they translate intersection
of cycles and Cartier divisors into evaluation. In precise terms, let
$X\in Z_{r}(\P^\bfn_K)$ be a cycle of pure dimension $r$, $ \bfd=(\bfd_{0},\dots,
\bfd_{r})\in (\N^{m+1}\setminus \{\bfzero\})^{r+1}$, and $f\in
K[\bfx]_{\bfd_{r}}$ such that $\div(f)$ intersects $X$ properly. Then
\begin{equation*}
\Res_{\bfd_0,\bfd_1\dots,\bfd_r}(X)(\bfu_0,\dots,\bfu_{r-1},f)
=\lambda \, \Res_{\bfd_0,\dots,\bfd_{r-1}}(X\cdot
\div(f))(\bfu_0,\dots,\bfu_{r-1}),  
\end{equation*}
with $\lambda \in K^\times$, see \cite[Proposition 3.6]{Remond2001:em} or
\cite[Proposition 1.40]{DKS2011:hvmsan}. 

Resultants also behave well with respect to other geometric
constructions including linear projections, products and ruled joins.
Both eliminants and resultants are invariant under  index permutations
and field extensions. The partial degrees of a resultant are
given by the mixed degrees of the underlying cycle, a fact already
exploited in the proof of Proposition \ref{prop:2}. The statements of
these properties and their proofs can be found in \cite{Remond2001:em,
  DKS2011:hvmsan}.

\subsection{Multiprojective toric varieties and
  cycles}\label{sec:toric-vari-cycl}
In this subsection, we set the standard notation for multiprojective
toric varieties and cycles, and prove some preliminary results, most
notably a formula for the intersection of a multiprojective toric
cycle and a toric Cartier divisor (Proposition \ref{prop:11}). We
assume a basic knowledge of the theory of normal toric varieties as
explained in \cite{Fulton:itv, CLS11}.

Let $n\ge 0$ and $M\simeq \Z^{n}$ a lattice of rank $n,$ and
set $N=M^{\vee}= \Hom(M,\Z)$ for its dual lattice.  Set also $M_{\R}=M\otimes\R$
and $N_{\R}=N\otimes\R$. The pairing between $x\in M_{\R}$ and $u\in
N_{\R}$ is denoted by $\langle x,u\rangle$. 

For a field $K$, we set
\begin{equation}\label{eq:94}
\T_{M,K}=\Spec(K[M])  
\end{equation}
for the \emph{(algebraic) torus} over $K$ corresponding to $M$. For
simplicity, we will focus on the case $K=\K$ is algebraically closed,
although all notions and results in this subsection are valid, with
suitable modifications, over an arbitrary field.  In our situation, we
write $\T_{M}=\T_{M,\K}$ for short. Since $\K$ is algebraically
closed, we can identify this torus with its set of points. With this
identification,
\begin{displaymath}
  \T_{M}=\Hom(M,\K^{\times})= N \otimes \K^{\times}\simeq
(\K^{\times})^{n}
\end{displaymath}
For $a\in M$, we denote by $\chi^{a} \colon \T_{M}\to \K^{\times}$ the
corresponding group homomorphism or \emph{character} of $\T_{M}$.

For $m\ge 0,$ consider a family of nonempty finite subsets
$\cA_{i}=\{a_{i,0},\dots, a_{i,c_{i}}\}\subset M$, $i=0,\dots, m$, and
set $ \bfcA=(\cA_{0}, \dots, \cA_{m})$.  Set  $\bfc=(c_{0},\dots,
c_{m})$ and consider the associated multiprojective space over $\K$
\begin{displaymath}
\P^{\bfc}=  \P^{\bfc}_{\K}=\prod_{i=0}^{m}\P^{c_{i}}_{ \K}.
\end{displaymath}
For each $i$, we denote by $\bfx_{i}=\{x_{i,0},\dots, x_{i,c_{i}}\}$ a
set of $c_{i}+1$ variables and we put $\bfx=\{\bfx_{1},\dots,
\bfx_{m}\}$, so that $\K[\bfx]=\K[\bfx_{1},\dots, \bfx_{m}]$ is the
multihomogeneous coordinate ring of $\P^{\bfc}$.

Let  $\varphi_{\bfcA}\colon \T_{M}\rightarrow
\P^{\bfc}$ be the monomial map given, for $\xi\in \T_{M}$, by
\begin{equation}\label{eq:23}
\varphi_{\bfcA}(\xi)= ((\chi^{a_{0,0}}(\xi):\dots:\chi^{a_{0,c_{0}}}(\xi)),
\dots, (\chi^{a_{m,0}}(\xi):\dots:\chi^{a_{m,c_{m}}}(\xi))).
\end{equation}
We then set
\begin{equation}
  \label{eq:27}
X_{\bfcA}= \ov{\varphi_{\bfcA}(\T_{M})}, \quad 
Z_{\bfcA}=(\varphi_{\bfcA})_{*}\T_{M}   
\end{equation}
for the associated multiprojective toric subvariety and toric cycle, respectively. 

For $i=0,\dots, m$, consider the sublattice of $M$ given by
\begin{equation}
  \label{eq:36}
L_{\cA_{i}}=\sum_{j=1}^{c_{i}}(a_{i,j}-a_{i,0})\Z,  
\end{equation}
and put $L_{\bfcA}=\sum_{i=0}^{m} L_{\cA_{i}}$.  By \cite[Proposition
1.1.8]{CLS11}, it follows that
\begin{equation}
  \label{eq:39}
  \dim(X_{\bfcA})=\rank(L_{\bfcA}).
\end{equation}
In particular, $X_{\bfcA}$ coincides with the support of $Z_{\bfcA}$
if and only if $\rank(L_{\bfcA})=n$. Otherwise, $\dim(X_{\bfcA}) \le n-1$
and $Z_{\bfcA}=0$.

For $i=0,\dots, m$, consider the convex hull
\begin{equation*}
\Delta_{i}=\conv(\cA_{i})\subset M_{\R}  .
\end{equation*}
It is a lattice polytope lying in a translate of the linear space
$L_{\cA_{i},\R}=L_{\cA_{i}}\otimes \R$.  We also set
\begin{math}
 \Delta=\sum_{i=0}^{m}\Delta_{i}   
\end{math}
for its Minkowski sum, which is a lattice polytope lying in a translate
of~$L_{\bfcA,\R}=L_{\bfcA}\otimes \R$.  We denote by $\Sigma_{\Delta}$ the conic
polyhedral complex on $N_{\R}$ given by the inner directions of
$\Delta$ as in \cite[page 26]{Fulton:itv} or \cite[Proposition
6.2.3]{CLS11}. If $\dim(\Delta)=n$, then $\Sigma_{\Delta}$ is a fan.

The multiprojective toric variety $X_{\bfcA}$ is not necessarily
normal. The next lemma shows that we can construct a proper normal
toric variety dominating it by considering any fan refining
$\Sigma_{\Delta}$. As it is customary, we denote by $X_{\Sigma}$ the
normal toric variety over $\K$ corresponding to a fan $\Sigma$ on
$N_{\R}$.

\begin{lemma} \label{lemm:11} Let $\Sigma$ be a fan in $N_{\R}$
  refining $\Sigma_{\Delta}$. The  map $\varphi_{\bfcA}$
  in~\eqref{eq:23} extends to a morphism of proper toric varieties
\begin{equation}\label{eq:15}
  \Phi_{\bfcA}\colon X_{\Sigma}\longrightarrow \P^{\bfc}.
\end{equation} 
In particular, $X_{\bfcA}=\Phi_{\bfcA}(X_{\Sigma})$ and
$Z_{\bfcA}=(\Phi_{\bfcA})_{*}X_{\Sigma} $.
\end{lemma}

\begin{proof}
  Let $\Sigma^{c_{i}}$ be the normal fan of the standard simplex
  of~$\R^{c_{i}}$, $i=0,\dots, m$, and set
  $\Sigma^{\bfc}=\prod_{i=0}^{r}\Sigma^{c_{i}}$, which is a fan on
  $\R^{\bfc}$. For each $i$, the toric variety associated to
  $\Sigma_{i}$ is $\P^{c_{i}}$ and so, by \cite[Proposition
  3.1.14]{CLS11}, the toric variety associated to $\Sigma^{\bfc}$ is
  the multiprojective space $\P^{\bfc}$.

The map
\begin{displaymath}
  (\K^{\times})^{|\bfc|}\longrightarrow \P^{\bfc},\quad
  (\bfz_{0},\dots,\bfz_{m})\longmapsto   ((1:\bfz_{0}),\dots,(1:\bfz_{m}))
\end{displaymath}
gives an isomorphism between the torus $(\K^{\times})^{|\bfc|}$ and
the open orbit $\P^{\bfc}_{0}$ of $\P^{\bfn}$. The image of
$\varphi_{\bfcA}$ is contained in this orbit and the map
$\varphi_{\bfcA}\colon \T_{M}\rightarrow \P^{\bfc}_{0}$ is a
homomorphism of tori. Under the correspondence in \cite[Theorem
3.3.4]{CLS11}, this homomorphism corresponds to the linear map
$A\colon N\to \Z^{\bfc}$ given, for $u\in N$, by
\begin{equation}\label{eq:58}
  A(u)= (\langle a_{i,j}-a_{i,0},u\rangle)_{0\le i\le m,1\le j\le c_{i}}.
\end{equation}

We have that $A^{-1}(\Sigma^{\bfc})=\Sigma_{\Delta}$. Since $\Sigma$
refines $\Sigma_{\Delta}$, it follows that this linear map is
compatible with the fans $\Sigma$ and $\Sigma^{\bfc}$ in the sense of
\cite[Definition~3.3.1]{CLS11}. By Theorem 3.3.4(a) in {\it
  loc. cit.}, $\varphi_{\bfcA}$ extends to a proper toric map
$\Phi_{\bfcA}\colon X_{\Sigma}\rightarrow \P^{\bfc}$.

Since $\Phi_{\bfcA}$ is a map of proper toric varieties and $\T_{M}$
is a dense open subset of $X_{\Sigma}$, 
\begin{displaymath}
  \Phi_{\bfcA}(X_{\Sigma})=\ov{\varphi_{\bfcA}(\T_{M})}= X_{\bfcA},
  \quad   (\Phi_{\bfcA})_{*}X_{\Sigma}=(\varphi_{\bfcA})_{*}\T_{M}= Z_{\bfcA},
\end{displaymath}
which completes the proof. 
\end{proof}

\begin{lemma} \label{lemm:1} Let notation be as in \eqref{eq:23}
  and \eqref{eq:27}. Then
\begin{displaymath}
  X_{\bfcA}\setminus \bigcup_{i=0}^{n}\bigcup_{j=0}^{c_{i}}V(x_{i,j})= \varphi_{\bfcA}(\T_{M}).
\end{displaymath}
\end{lemma}

\begin{proof}
  By translating the subsets $\cA_{i}$ and restricting them to the
  sublattice~$L_{\bfcA}$, we can reduce without loss of generality to
  the case when $M=L_{\bfcA}$.  Assume that we are in this situation,
  and consider the morphism of proper toric varieties $
  \Phi_{\bfcA}\colon X_{\Sigma}\longrightarrow \P^{\bfc} $
  in~\eqref{eq:15} and its associated linear map $A\colon
  N\to\Z^{\bfc}$ as in \eqref{eq:58}.  For each cone $\sigma\in
  \Sigma$ we denote by $O(\sigma)$ the associated orbit under the
  orbit-cone correspondence explained in \cite[\S3.1]{Fulton:itv} and
  \cite[\S3.2]{CLS11}.

  The correspondence $\sigma\mapsto O(\sigma)$ is a bijection and so
  there is a decomposition
\begin{equation*}
 X_{\Sigma}=\bigsqcup_{\sigma\in \Sigma}O(\sigma) .
\end{equation*}

We have that $O(0)=\T_{M}$ and
$\Phi_{\bfcA}(O(0))=\varphi_{\bfcA}(\T_{M})$ is contained in $
\P^{\bfc}_{0}$, the open orbit of $\P^{\bfc}$. On the other hand, the
hypothesis that $M=L_{\bfcA}$ implies that the linear map $A$ is
injective and so, given $\sigma\in \Sigma\setminus \{0\}$, we have
that $A(\sigma)\ne 0$.  By \cite[Lemma 3.3.21(b)]{CLS11},
$\Phi_{\bfcA}(O({\sigma})) $ is contained in $\P^{\bfc}\setminus
\P^{\bfc}_{0}= \bigcup_{i,j}V(x_{i,j})$ . It follows that
\begin{displaymath}
  X_{\bfcA}\setminus \bigcup_{i,j}V(x_{i,j})=
\Big(  \bigcup_{\sigma\in\Sigma}\Phi_{\bfcA}(O(\sigma))\Big)\setminus \bigcup_{i,j}V(x_{i,j})= \Phi_{\bfcA}(O(0))=\varphi_{\bfcA}(\T_{M}),
\end{displaymath}
as stated. 
\end{proof}

Now suppose that the lattice polytope $\Delta$ has dimension $n$ and
let $\Gamma$ be a facet, that is, a face of $\Delta$ of codimension
1. Let $L_{\Gamma\cap M}\simeq \Z^{n-1}$ be the sublattice of $M$
generated by the differences of the lattice points of $\Gamma$ and
$\T_{L_{\Gamma\cap M}}\simeq (\K^{\times})^{n-1}$ its associated
torus.  Let $v({\Gamma})\in N$ denote the primitive inner normal vector
of $\Gamma$ and, for each $i$, set $\Gamma_{i}$ for the face of
$\Delta_{i}$ which minimizes the functional $v({\Gamma})\colon
M_{\R}\to \R$ on $\Delta_{i}$.

We consider the morphism $\varphi_{\bfcA,\Gamma}\colon \T_{L_{\Gamma\cap M}}\to
\P^{\bfc}$ given, for $\xi\in  \T_{L_{\Gamma\cap M}}$, by
\begin{equation} \label{eq:1}
  \varphi_{\bfcA,\Gamma}(\xi)_{i,j}=
  \begin{cases}
    \chi^{a_{i,j}}(\xi)& \text{ if } a_{i,j}\in \Gamma_{i}, \\
0 & \text{ otherwise.}
  \end{cases}
\end{equation}
Set $Z_{\bfcA,\Gamma}=(\varphi_{\bfcA,\Gamma})_{*}(\T_{L_{\Gamma\cap
    M}})\in Z_{n-1}(\P^{\bfc})$ for the associated multiprojective
toric cycle.

For a bounded subset $P\subset M_{\R}$, we define its \emph{support
  function} as the function $h_{P}\colon N_{\R}\to \R$ given, for $v\in
N_{\R}$, by
\begin{equation}\label{eq:18}
h_{P}(v)= \inf_{x\in P}\langle v,x\rangle. 
\end{equation}
The usual convention in convex analysis is to define support functions
as convex functions by putting a ``$\sup$'' instead of the ``$\inf$''
in the formula above as it is done, for instance, in \cite[page
37]{Schneider:cbbmt}. Our notion of support function gives a
\emph{concave} function, and is better suited to toric geometry.

\begin{proposition} \label{prop:11} Let notation be as above and let
$0\le i\le m$ and $0\le j\le c_{i}$. If $\dim(\Delta)=n$, then
\begin{equation}\label{eq:88}
  Z_{\bfcA}\cdot \div(x_{i,j})= \sum_{\Gamma}-h_{\Delta_{i}-a_{i,j}}(v({\Gamma})) \, Z_{\bfcA,\Gamma},
\end{equation}
where the sum is over the facets $\Gamma$ of $\Delta$. Otherwise,
$Z_{\bfcA}\cdot \div(x_{i,a_{i,j}})=0$.
\end{proposition}

\begin{proof}
  By symmetry, we can suppose without loss of generality that $i=j=0$.
  Consider first the case when $\dim(\Delta)=n$. Then
  $\Sigma_{\Delta}$ is a fan and so, by Lemma \ref{lemm:11}, the map
  $\varphi_{\bfcA}$ extends to a morphism of proper toric varieties
  \begin{equation*}
\Phi_{\bfcA}\colon  X_{\Sigma_{\Delta}}\longrightarrow X_{\bfcA}
  \end{equation*}
and $Z_{\bfcA}=(\Phi_{\bfcA})_{*}X_{\Sigma_{\Delta}}$. 

Set $D=(\Phi_{\bfcA})^{*}\div(x_{0,0})\in \Div(X_{\Sigma_{\Delta}})$.
By the projection formula \eqref{eq:38},
\begin{equation}\label{eq:60}
Z_{\bfcA}\cdot \div(x_{0,0})= (\Phi_{\cA})_{*}(X_{\Sigma_{\Delta}}\cdot D).
\end{equation}

Let $\Psi_{D}\colon N_{\R}\to \R$ be the virtual support function of
$D$ under the correspondence in \cite[Theorem 4.2.12]{CLS11}. Using either
Theorem 4.2.12(b) in {\it loc. cit.} or \cite[Lemma, page
61]{Fulton:itv}, it follows that
\begin{equation}\label{eq:61}
X_{\Sigma_{\Delta}}\cdot D= \sum_{\tau} -\Psi_{D}(v_{\tau}) V(\tau),
\end{equation}
the sum being over the rays $\tau$ of ${\Sigma_{\Delta}}$, where
$v_{\tau}$ denotes the first nonzero vector in $\tau \cap N$ and
$V(\tau)$ denotes the $\T_{M}$-invariant prime Weil divisor of
$X_{\Sigma_{\Delta}}$ determined by~$\tau$.

Let $\Delta^{c_{0}}=\conv(\bfzero,\bfe_{0,1},\dots, \bfe_{0,c_{0}})$
be the standard simplex of $\R^{c_{0}}$ and $\Delta^{c_{0}}\times
\{\bfzero\}$ its immersion into $ \R^{\bfc}$. We can verify that the
virtual support function associated to the Cartier divisor
$\div(x_{0,0}) \in\Div(\P^{\bfc})$ under the correspondence in
\cite[Theorem~4.2.12]{CLS11} coincides with $h_{\Delta^{c_{0}}\times
  \{\bfzero\}}$, the support function of this polytope.  By {\it
  loc. cit}, Proposition~6.2.7, $\Psi_{D}= h_{\Delta^{c_{0}}\times
  \{\bfzero\}}\circ A$ where $A\colon N\to\Z^{\bfc}$ denotes the
linear map in~\eqref{eq:58}. This implies that
\begin{equation}\label{eq:62}
\Psi_{D}=h_{\Delta_{0}-a_{0,0}},
\end{equation}
the support function of the translated polytope ${\Delta_{0}-a_{0,0}}\subset
M_{\R}$.

By construction, the rays of $\Sigma_{\Delta}$ are the inner normal
directions of the facets of~$\Delta$.  For each ray $\tau$, the prime
Weil divisor $V(\tau)$ is the closure of the orbit $O(\tau)$
associated to $\tau$ under the orbit-cone correspondence.  We denote
by $\tau^{\bot}$ the subspace of $M_{\R}$ orthogonal to $\tau$ and by
$\iota_{\tau}\colon \T_{\tau^{\bot}\cap M}\to O(\tau)$ the isomorphism
in \cite[Lemma 3.2.5]{CLS11}.

Let $\Gamma$ be the facet of $\Delta$ corresponding to $\tau$. Hence,
$v_{\tau}=v({\Gamma})$, the primitive inner normal vector of
$\Gamma$. We can verify that $\tau^{\bot}\cap M=L_{\Gamma\cap M}$ and
so $\T_{\tau^{\bot}\cap M}=\T_{L_{\Gamma\cap M}} $, and that the
composition $\Phi_{\bfcA}\circ \iota_{\tau}$ coincides with the map
$\varphi_{\bfcA,\Gamma}$ in \eqref{eq:1}.  Hence
\begin{equation}
  \label{eq:87}
(\Phi_{\bfcA})_{*}V(\tau)=Z_{\bfcA,\Gamma}.   
\end{equation}
The formula \eqref{eq:88} then follows from \eqref{eq:60},
\eqref{eq:61}, \eqref{eq:62} and \eqref{eq:87}.

In the case when $\dim(\Delta)<n$, we have that $\rank(L_{\bfcA})<n$.
It follows that $Z_{\bfcA}=0$ by \eqref{eq:39} and, {\it a fortiori},
that $Z_{\bfcA}\cdot \div(x_{0,0})=0$.
\end{proof}

\subsection{Root counting on algebraic tori} \label{sec:root-count-algebr}
     
It is well-known that the number of roots of a family of Laurent
polynomial is related to the combinatorics of the exponents appearing
in its monomial expansion. For the convenience of the reader, we
recall the results in this direction that we will use in the sequel.

We denote by $\vol_{M}$  the Haar measure on $M_{\R}$ normalized so
that $M$ has covolume~1. 
The \emph{mixed volume} of a family of compact bodies $Q_1,
\dots, Q_n\subset M_{\R}$ is defined as
\begin{equation} \label{defmultivolume}
\MV_M(Q_1, \dots, Q_n) = \sum_{j=1}^n (-1)^{n - j} \sum_{1 \le i_1 < \cdots < i_j \le n} \vol_M(Q_{i_1} + \cdots + Q_{i_j}).
\end{equation} 
For $n=0$ we agree that $ \MV_{M}=1$.     

 We have that $\MV_M(Q, \dots, Q) = n! \vol_M(Q)$ and so the mixed
 volume can be seen as a generalization of the volume of a convex
 body.  The mixed volume is symmetric and linear in each variable
 $Q_i$ with respect to the Min\-kows\-ki sum, invariant with respect
 to isomorphisms of lattices, and monotone with respect to the
 inclusion of compact bodies of $M_{\R}$, see for instance
 \cite[\S7.4]{CLO05} or \cite [Chapter 5]{Schneider:cbbmt}.

 Let $K$ be a field and $\ov K$ its algebraic closure. Given a square
 family of Laurent polynomials $f_{i}\in K[M]$, $i=1,\dots, n$, we
 denote by $Z(f_{1},\dots,f_{n})$ the cycle on $\T_{M,\ov K}$, given
 by its isolated roots together with their corresponding
 multiplicities. In precise terms,
\begin{equation}
  \label{eq:63}
  Z(f_{1},\dots,f_{n})= \sum_{\xi} m_{\xi} \, \xi,
\end{equation}
the sum being over the isolated points $\xi$ of $V(f_{1},\dots,
f_{n})\subset \T_{M,\ov K}$ and where, if $I(\xi)\subset \ov K[M]$
denotes the ideal of $\xi$, the multiplicity $m_{\xi}$ is given by
\begin{equation}\label{eq:86}
m_{\xi}=\dim_{\ov K}( \ov K[M]/(f_{1},\dots, f_{n})) _{I({\xi})}.
\end{equation}

Write 
\begin{equation}\label{eq:57}
  f_{i}=\sum_{j=0}^{c_{i}}\alpha_{i,j}\chi^{a_{i,j}}, \quad i=1,\dots, n,
\end{equation}
with $\alpha_{i,j}\in K^{\times}$ and $a_{i,j}\in M$. The \emph{Newton
  polytope} of $f_{i}$ is  given by
\begin{equation*}
 \Delta_{i}=\newton(f_{i})=\conv (a_{i,0},\dots, a_{i,c_{i}})\subset M_{\R}.
\end{equation*}

For $v\in N_{\R}$, we denote by $\Delta_{i,v}\subset M_{\R}$ the
subset of points of $\Delta_{i}$ whose weight in the direction of $v$
is minimal. It is a face of $\Delta_{i}$. We also set
\begin{displaymath}
  f_{i,v}=\sum_{j}\alpha_{i,j}\chi^{a_{i,j}}\in K[M], \quad i=1,\dots,n,
\end{displaymath}
the sum being over $0\le j\le c_{i}$ such that $a_{i,j}\in
\Delta_{i,v}$. 

Bernstein's theorem {\cite[Theorem~B]{Ber75}} states that, if
$\car(K)=0$ and, for all $v\in N\setminus \{0\}$, the family
$f_{i,v}$, $i=1,\dots, n$, has no root in $\T_{M,\ov K}$, then
$V(f_{1},\dots, f_{n})$ is finite and
 \begin{equation}\label{eq:46}
   \deg(Z(f_{1},\dots,f_{n})) =   \MV_{M}(\Delta_{1},\dots, \Delta_{n}).
 \end{equation}
This statement also holds for an arbitrary field $K$ \cite[Proposition
1.4]{PS:rbke}.

When $K$ is endowed with a discrete valuation $\val\colon K^{\times}\to
 \R$, there is a
 refinement of Bernstein's theorem due to Smirnov
 \cite{Smi97:tsodvr}, that gives a combinatorial expression for the
 number of roots with a given valuation.

 To state it properly, let $K^{\circ}$ and $K^{\circ\circ}$ denote the
 valuation ring and its maximal ideal associated to the pair
 $(K,\val)$. Let $\kappa$ be a uniformizer of $K^{\circ}$, that is, a
 generator of $K^{\circ\circ}$, and $k=K^{\circ}/K^{\circ\circ }$ the
 residue field. For $\alpha\in K$, the \emph{initial part} of $\alpha$
 with respect to $\kappa$, denoted by $ \init_{\kappa}(\alpha) $, is
 defined as the class in $k$ of the element
 $\kappa^{-\val(\alpha)}\alpha\in K^{\circ}$.

 Consider also an arbitrary extension of the valuation to $\ov
 K$. Since $\T_{M,\ov K}=N_{\R}\otimes \ov K^{\times}$, this valuation
 induces a map $\T_{M,\ov K}\rightarrow N_{\R}$, that we also denote
 by $\val$. For a square family of Laurent polynomials as before and
 $w\in N_{\R}$, we consider  the cycle on $\T_{M,\ov K}$ given by
\begin{equation*}
  Z(f_{1},\dots,f_{n})_{w}= \sum_{\xi} m_{\xi} \, \xi,
\end{equation*}
the sum being over the isolated points $\xi$ of $V(f_{1},\dots,
f_{n})\subset \T_{M,\ov K}$ such that $\val(\xi)=w$, and with
multiplicities $m_{\xi}$ as in \eqref{eq:63}.

For $i=1,\dots, n$, we consider the lifted polytope of $f_{i}$ defined
as
\begin{equation}\label{eq:26}
  \wt \Delta_{i}=\conv ((a_{i,0}, -\val(\alpha_{i,0})),\dots,
  (a_{i,c_{i}},-\val(\alpha_{i,c_{i}})))\subset M_{\R} \times \R.
\end{equation}
Given $w\in N_{\R}$, we denote by $ \wt \Delta_{i,(w,1)}\subset
M_{\R}\times \R$ the subset of points of $\wt \Delta_{i}$ whose weight
in the direction of $(w,1)$ is minimal. It is a face of this lifted
polytope contained in its upper envelope. Then we set
$\Delta_{i,(w,1)}\subset M_{\R}$ for the image of this face under
the projection $M_{\R}\times \R\to M_{\R}$. We also set
\begin{displaymath}
  f_{i,(w,1)}=\sum_{j}\init_{\kappa}(\alpha_{i,j})\chi^{a_{i,j}}
  \in k[M], \quad i=1,\dots, n,
\end{displaymath}
the sum being over $0\le j\le c_{i}$ such that
$(a_{i,j},-\val(\alpha_{i,j}))\in \wt \Delta_{i,(w,1)}$. 

In this situation, Smirnov's theorem \cite[Theorem
3.2.2(b)]{Smi97:tsodvr} states that if, for all $w\in N$ such that
$\dim(\sum_{i=1}^{n}\wt \Delta_{i,(w,1)})<n$ the family of Laurent
polynomials $f_{i,(w,1)}\in k[M]$, $i=1,\dots, n$, has no root in
$\T_{M,\ov k}$, then, for any $w_{0}\in N_{\R}$, the set of points of
$V(f_{1},\dots,f_{n})$ with valuation $w_{0}$ is finite and
 \begin{equation}\label{eq:102}
\deg(Z(f_{1},\dots,f_{n})_{w_{0}})=\MV_{M}(\Delta_{1,(w_{0},1)},\dots, \Delta_{n,(w_{0},1)}).
 \end{equation}

 We are interested in the following generic situation. For $i=1,\dots,
 n$, let $\bfu_{i}=\{u_{i,0}, \dots, u_{i,c_{i}}\} $ be a set of
 $c_{i}+1$ variables and set $\ov\bfu=\{\bfu_{1},\dots, \bfu_{n}\}$.
 For a polynomial $R=\sum_{\bfb}\beta_{\bfb}\ov \bfu^{\bfb}\in K[\ov
 \bfu]$, we set
\begin{equation} \label{eq:103}
  \val(R)= \min_{\bfb}\val (\beta_{\bfb}).
\end{equation}
By Gauss' lemma, this gives a discrete valuation on the field
$\F:=K(\ov \bfu)$ that extends $\val$. We then consider an arbitrary
extension of this valuation to the algebraic closure~$\ov{\F}$ and the
associated map $\val\colon \T_{M,\ov \F}\to N_{\R}$ as before. We denote
by $\mathfrak{f}$ the residue field of~$\F$. 

\begin{proposition} \label{prop:8} With notation as above, set
  \begin{displaymath}
   F_{i}=\sum_{j=0}^{c_{i}}u_{i,j} \chi^{a_{i,j}}\in K[\bfu_{i}][M],
   \quad i=1,\dots, n.     
  \end{displaymath}
  Then $V(F_{1},\dots, F_{n})\subset\T_{M,\ov \F}$ is finite, $\deg(Z(F_{1},\dots,
  F_{n}))=\MV_{M}(\Delta_{1},\dots, \Delta_{n})$, and $\val(\xi)=0$ for
  all $\xi\in V(F_{1},\dots, F_{n})$.
\end{proposition}

\begin{proof}
Let $v\in N\setminus \{0\}$. Since $\Delta_{i,v}$ lies in a
translate of the orthogonal space $v^{\bot}$, the roots of in the
torus of the system
 $F_{i,v}$, $i=1,\dots,n$, are the roots of an equivalent system of
 $n$ general Laurent polynomials in $n-1$ variables. Hence, this set
 of roots is empty, and Bernstein's theorem \eqref{eq:46} implies the
 first and the second claims.

 For the last claim, denote by $\wt{\Delta_{i}} \subset M_{\R}\times
 \R$ the lifted polytope in associated to $F_{i}$ as in \eqref{eq:26},
 $i=1,\dots, n$. Since $\val(u_{i,j})=0$ for all $j$, we have that $
 \wt\Delta_{i,(w,1)}= \Delta_{i,w}\times \{0\} $. We deduce that, for $w\in N$,
 $ \wt\Delta_{i,(w,1)}= \Delta_{i,w}\times \{0\} $ and $F_{i,(w,1)}$
 coincides with the class of $F_{i,w}$ in the polynomial ring
 $\mathfrak{f}[\ov \bfu]$.

Suppose now that $\dim(\sum_{i=1}^{n}\wt \Delta_{i,(w,1)})<n$. Then
$\dim(\sum_{i=1}^{n} \Delta_{i,w})<n$ and, similarly as before, the
system $F_{i,(w,1)}$, $i=1,\dots, n$, has no roots in $\T_{M,\ov{  \mathfrak{f}}}$.
Smirnov's theorem applied to the case when  $f_{i}=F_{i}$, $i=1,\dots,
n$, and $w_{0}=0$, implies that 
\begin{displaymath}
\deg(Z(F_{1},\dots,
  F_{n})_{0})=\MV_{M}(\Delta_{1},\dots, \Delta_{n})=\deg(Z(F_{1},\dots,
  F_{n})).
\end{displaymath}
Hence all points of  $V(F_{1},\dots, F_{n})$ have valuation 0, which
concludes the proof.
\end{proof}

 \section{Basic properties of sparse eliminants and
   resultants} \label{sec:basic-prop-sparse}

 In this section, we show that the sparse eliminant and the sparse
 resultant respectively coincide with the eliminant and the resultant
 of a multiprojective toric variety/cycle. Using this interpretation, we
 derive some of their basic properties from the corresponding ones for
 general eliminants and resultants.

 We will freely use the notation in \S\ref{sec:toric-vari-cycl} with
 $K=\Q$ and $\K=\C$. We also set $m=n$ so that, in particular, we have that
 $\cA_{i}=\{a_{i,0},\dots, a_{i,c_{i}}\}$, $i=0,\dots, n$, is a family
 of $n+1$ nonempty finite subsets of~$M$ or \emph{supports}. We denote
 by $\Delta_{i}=\conv(\cA_{i})$ the convex hull of $\cA_{i}$.

For $i=0,\dots, n$, let $\bfu_{i}=\{u_{i,0}, \dots, u_{i,c_{i}}\} $ be
a set of $c_{i}+1$ variables. Set $\bfu=\{\bfu_{0},\dots, \bfu_{n}\}$,
so that $\C[\bfu]=\C[\bfu_{0},\dots, \bfu_{n}]$ is the
multihomogeneous coordinate ring of the multiprojective space
\begin{displaymath}
\P^{\bfc}=\prod_{i=0}^{n}\P^{c_{i}}_{\C}.  
\end{displaymath}
For each $i$, we consider 
 the general Laurent polynomial with support
 $\cA_{i}$ given by
\begin{equation}\label{eq:24}
  F_{i}=\sum_{j=0}^{c_{i}}u_{i,j} \chi^{a_{i,j}}\in \Q[\bfu_{i}][M]. 
\end{equation}
We set for short
\begin{equation}\label{eq:98}
   \bfcA=(\cA_{0},\dots, \cA_{n}), \quad 
  \Delta=\sum_{i=0}^{n}\Delta_{i} \quad \text{ and } \quad 
  \bfF=(F_{0},\dots, F_{n})
\end{equation}

The incidence variety of the family $\bfF$ is 
\begin{equation*} 
  \Omega_{\bfcA}=\{(\xi,\bfu)\mid
  F_{0}(\bfu_{0},\xi)=\dots=F_{n}(\bfu_{n},\xi)=0\} \subset \T_{M}\times \P^{\bfc}, 
\end{equation*}
which is an irreducible subvariety of codimension $n+1$ defined over
$\Q$. We denote by  $\pi\colon \T_{M}\times \P^{\bfc}\to \P^{\bfc}$  the
projection onto the second factor.

\begin{definition} \label{def:2} The \emph{$\bfcA$-eliminant} or
  \emph{sparse eliminant}, denoted by $\Elim_{\bfcA} $, is defined as
  any irreducible polynomial in $\Z[\bfu]$ giving an equation for the
  closure of the image
  $\ov{\pi(\Omega_{\bfcA})}$, if this  is a hypersurface,
  and as $1$ otherwise.
  
  The \emph{$\bfcA$-resultant} or \emph{sparse resultant}, denoted by
  $\Res_{\bfcA}$, is defined as any primitive polynomial in $\Z[\bfu]$
  giving an equation for the direct image
  $\pi_{*}\Omega_{\bfcA}$.
\end{definition}

Both the sparse eliminant and the sparse resultant are well-defined up
to a sign.  It follows from these definitions that there exists ${d_{\bfcA}}
\in \N$ such that
\begin{equation} \label{eq:77}
  \Res_{\bfcA}=\pm\Elim_{\bfcA}^{d_{\bfcA}},
\end{equation}
with $d_{\bfcA}$ equal to the  degree of the restriction of $\pi$
to the incidence variety  $\Omega_{\bfcA}$.

Let $Z_{\bfcA}$ be the multiprojective toric cycle on $\P^{\bfc}$ as in \eqref{eq:27}
and $|Z_{\bfcA}|$ its support. Both are defined over $\Q$, and we will
consider their eliminants and resultants, in the sense of Definitions
\ref{def:3} and \ref{def:6}, with respect to the ring $A=\Z$.

\begin{proposition} \label{prop:1} Let notation be as before and set
  $\bfe_{i}$, $i=0,\dots, n$, for the standard basis of
  $\Z^{n+1}$. Then
\begin{displaymath}
  \Elim_{\bfcA}=\pm \Elim_{\bfe_{0}, \dots, \bfe_{n}}(|Z_{\bfcA}|) \quad
  \text{ and } \quad \Res_{\bfcA}=\pm \Res_{\bfe_{0}, \dots, \bfe_{n}}(Z_{\bfcA}). 
\end{displaymath}
\end{proposition}

\begin{proof} 
  Let $\bfx=\{x_{i,j}\}_{i,j}$ and $\bfu=\{u_{i,j}\}_{i,j}$
  respectively denote the homogeneous coordinates of the first and the
  second factor in the product $\P^{\bfc}\times \P^{\bfc}$,
  respectively. For $i=0,\dots, n$, consider the general linear form
  on $\P^{c_{i}}$ given by
\begin{equation}\label{eq:71}
L_{i}=\sum_{j=0}^{c_{i}} u_{i,j} x_{i,j} \in \Q[\bfu][\bfx_{i}].
\end{equation}

Let $\Sigma$ be a fan refining $\Sigma_{\Delta}$ and
$\Phi_{\bfcA}\colon X_{\Sigma}\to \P^{\bfc}$ the corresponding
morphism of proper toric varieties as in Lemma \ref{lemm:11}.
For each $i$,  set 
\begin{equation*}
  D_{i}=(\Phi_{\bfcA}\times
\Id_{\P^{\bfc}})^{*}(\div(L_{i}))\in \Div(X_{\Sigma}\times \P^{\bfc}). 
\end{equation*}
This is a Cartier divisor whose restriction to $\T_{M}\times
\P^{\bfc}$ coincides with $\div(F_{i})$ for the general Laurent
polynomial $F_{i}$ as in \eqref{eq:24}.

By Lemma \ref{lemm:11}, $Z_{\bfcA}\times\P^{\bfc}= (\Phi_{\bfcA}\times
\Id_{\P^{\bfc}})_{*}(X_{\Sigma}\times\P^{\bfc})$ and the family $\div(L_{i})$, $i=0,\dots, n$, intersects this cycle
properly. By the projection formula~\eqref{eq:38}, it follows that
\begin{displaymath}
(\Phi_{\bfcA}\times
\Id_{\P^{\bfc}})_{*}\Big((X_{\Sigma}\times\P^{\bfc})\cdot\prod_{i=0}^{n}D_{i}\Big)= (Z_{\bfcA}\times\P^{\bfc})\cdot
  \prod_{i=0}^{n} \div(L_{i}). 
\end{displaymath}
Let $\rho\colon \P^{\bfc}\times \P^{\bfc}\to \P^{\bfc}$ be the
projection onto the second factor as in \eqref{eq:90}. By the
functoriality of the direct image, $\pi_{*}= \rho_{*}\circ
(\Phi_{\bfcA}\times \Id_{\P^{\bfc}})_{*}$. Hence
\begin{equation} \label{eq:28}
  \pi_{*}\Big((X_{\Sigma}\times \P^{\bfc})\cdot\prod_{i=0}^{n}D_{i}\Big)= \rho_{*}\Big((Z_{\bfcA}\times\P^{\bfc})\cdot
  \prod_{i=0}^{n} \div(L_{i})\Big) =\rho_{*}\Omega_{Z_{\bfcA}, (\bfe_{0}, \dots, \bfe_{n})}
\end{equation}
for the incidence cycle $\Omega_{Z_{\bfcA}, (\bfe_{0}, \dots,
  \bfe_{n})}$ as in \eqref{eq:30}.

On the other hand, the general linear form $L_{i}$ does not vanish
identically on $\xi\times \P^{\bfc}$ for any $\xi\in
X_{\bfcA}$. Hence, the support of $D_{i}$ does not contain
$\zeta\times \P^{\bfc}$ for any $\zeta\in X_{\Sigma}$. This implies
that no component of the intersection cycle $(X_{\Sigma}\times
\P^{\bfc})\cdot\prod_{i=0}^{n}D_{i}$ is supported in
$(X_{\Sigma}\setminus \T_{M})\times \P^{\bfc}$. It follows that
\begin{equation}\label{eq:29}
  \pi_{*}\Big((X_{\Sigma}\times\P^{\bfc})\cdot\prod_{i=0}^{n}D_{i}\Big)= 
  \pi_{*}\Big((\T_{M}\times\P^{\bfc})\cdot\prod_{i=0}^{n}\div(F_{i})\Big) =\pi_{*}\Omega_{\bfcA}.
\end{equation}
From \eqref{eq:28} and \eqref{eq:29} we deduce the equality of cycles
$\rho_{*}\Omega_{Z_{\bfcA}, (\bfe_{0}, \dots,
  \bfe_{n})}=\pi_{*}\Omega_{\bfcA}$, which implies the statement
for the resultants and, {\it a fortiori},  for the eliminants.
\end{proof}

We devote the rest of this section to the study of the basic
properties of sparse eliminants and resultants.

\begin{proposition} \label{prop:5} Both the sparse eliminant and the sparse
  resultant are invariant, up to a sign, under permutations and
  translations of the supports.
\end{proposition}

\begin{proof}
  The first statement follows directly from Proposition \ref{prop:1}
  and \cite[Proposition 1.27]{DKS2011:hvmsan}. The second claim
  is a consequence of the fact that the monomial map $\varphi_{\bfcA}$
  in~\eqref{eq:23} is invariant under translations of the supports.
\end{proof}

The following proposition gives the partial degrees of the sparse
resultant. It is the analogue of the well-known formula for the
partial degrees of the sparse eliminant given in \cite[Chapter 8,
Proposition 1.6]{GKZ94} under some hypothesis, and by \cite[Corollary
2.4]{PS93} in the general case.

 \begin{proposition} \label{prop:6}
For $i=0,\dots, n$, 
\begin{equation*}
  \deg_{\bfu_{i}}(\Res_{\bfcA})= \MV_{M}(\Delta_{0}, \dots,
  \Delta_{i-1},\Delta_{i+1},\dots, \Delta_{n}),
\end{equation*}
where $\Delta_{i}\subset M_{\R}$ denotes the convex hull of $\cA_{i}$
and $\MV_{M}$ is the mixed volume of convex bodies as in
\eqref{defmultivolume}.
\end{proposition}

\begin{proof}
  By Proposition \ref{prop:1} and
  \cite[Proposition~1.32]{DKS2011:hvmsan},
\begin{equation}\label{eq:35}
  \deg_{\bfu_{i}}(\Res_{\bfcA})=   \deg_{\bfu_{i}}(\Res_{\bfe_{0},
    \dots, \bfe_{n}}(Z_{\bfcA}))= \deg\Big(Z_{\bfcA}\cdot\prod_{j\ne i}H_{j} \Big),
\end{equation}
where $H_{j}\subset\P^{\bfc}$ is the inverse image under the
projection $\P^{\bfc}\to \P^{c_{j}}$ of a generic hyperplane of $\P^{c_{j}}$.

Let $\Sigma$ be a fan refining $\Sigma_{\Delta}$ and
$\Phi_{\bfcA}\colon X_{\Sigma}\to \P^{\bfc}$ the morphism of proper
toric varieties as in Lemma \ref{lemm:11}.  For $j=0,\dots, n$, set
\begin{displaymath}
D_{j}=(\Phi_{\bfcA})^{*} H_{j}\in \Div(X_{\Sigma}).  
\end{displaymath}
Observe that the restriction of $D_{j}$ to $\T_{M}$ coincides with the
Cartier divisor of a generic Laurent polynomial $f_{j}\in \C[M]$ with
support $\cA_{j}$.  By the projection formula~\eqref{eq:38},
\begin{equation}\label{eq:33}
Z_{\bfcA}\cdot\prod_{j\ne i}H_{j} =
(\Phi_{\bfcA})_{*}\Big(X_{\Sigma}\cdot \prod_{j\ne i}\div(D_{j})\Big).  
\end{equation}
Since the hyperplanes $H_{j}$ are generic, the cycle $X_{\Sigma}\cdot
\prod_{j\ne i}\div(D_{j})$ is supported on $\T_{M}$ and so
\begin{equation}\label{eq:34}
X_{\Sigma}\cdot \prod_{j\ne i}\div(D_{j})=   \T_{M}\cdot \prod_{j\ne i}\div(f_{j}). 
\end{equation}
By Bernstein's theorem~\eqref{eq:46}, the degree of the cycle in the
right-hand side of \eqref{eq:34}
coincides with the mixed volume $\MV_{M}(\Delta_{0}, \dots,
\Delta_{i-1},\Delta_{i+1},\dots, \Delta_{n})$. The statement then
follows from \eqref{eq:35}, \eqref{eq:33} and~\eqref{eq:34}.
\end{proof}

We recall here the notion of essential subfamily of supports
introduced by Sturmfels in~\cite{Stu94}. For $J\subset\{0,\dots, n\}$,
we set
\begin{equation*}
  L_{\bfcA_{J}}=\sum_{j\in J}L_{\cA_{j}}
\end{equation*}
with $L_{\cA_{j}}$ as in \eqref{eq:36}.

\begin{definition} 
  \label{def:1} Let $J\subset \{0,\ldots,n\}$. The subfamily
  $\bfcA_{J}=(\cA_j)_{j\in J}$ is \emph{essential} if the following
  conditions hold:
\begin{enumerate}
 \item \label{item:14} $\#J=\rank\big(L_{\bfcA_J}\big)+1;$
\item \label{item:15} $\#J'\leq \rank\big(L_{\bfcA_J'}\big)$ for all $J'\subsetneq J$.
\end{enumerate}
\end{definition}

\begin{remark}\label{rem:1}
  When $J=\emptyset$, we have that $L_{\bfcA_{J}}=0$ and so $\#J=
  \rank\big(L_{\bfcA_J'}\big)=0$. In particular, if $\bfcA_{J}$ is an
  essential subfamily, then $J\ne\emptyset$.  On the other extreme,
  when the family $\bfcA$ is essential, $\bfcA_{J}$ is essential if
  and only if $J=\{0,\dots, n\}$.
\end{remark}

\begin{lemma}
  \label{lemm:6} Let $I\subset \{0,\dots, n\}$ 
  such that $ \rank(L_{\bfcA_{I}}) <\# I$.
  Then there exists $J\subset I$ such that $\bfcA_{J}$ is essential.
\end{lemma}

\begin{proof}
  Choose a  subset $J\subset I$ which is minimal with respect
  to the inclusion, under the condition that $\rank(L_{\bfcA_{J}}) <\#
  J$. Such a minimal subset exists because of the hypothesis
  that $ \rank(L_{\bfcA_{I}}) <\# I$. We have that
  $\rank(L_{\bfcA_{J'}}) \ge \#I$ for all $J'\subsetneq J$, and the
  minimality of $J$ implies that $\rank(L_{\bfcA_{J}}) =\# J-1$. Hence,
  $J$ is essential.
\end{proof}

The notion of essential subfamily gives a combinatorial
criterion to decide when $\Res_{\bfcA}\ne 1$ and, in that case, to
determine which are the sets of variables that actually appear in
the sparse eliminant and the sparse resultant.

 \begin{proposition} Let notation be as above. 
   \label{prop:3} 
   \begin{enumerate}
   \item \label{item:16} The following conditions are equivalent: 
     \begin{enumerate}
     \item \label{item:12}   $\Elim_{\bfcA}\neq1$; 
\item \label{item:19}  $\Res_{\bfcA}\neq1$;
\item \label{item:3} $ \rank(L_{\bfcA_{I}})\ge \#I -1$ for all $I\subset\{0,\dots, n\}$; 
\item \label{item:20}   there exists a unique  essential subfamily of $\bfcA$.
     \end{enumerate}
   \item \label{item:17} Suppose that $\Elim_{\bfcA}\ne 1$ or
     equivalently, that $\Res_{\bfcA}\ne 1$, and let $\bfcA_{J}$ be the
     unique essential subfamily of $\bfcA$. Then the following
     conditions are equivalent:
     \begin{enumerate}
     \item \label{item:21} $\deg_{\bfu_{i}}(\Elim_{\bfcA})>0$;
\item \label{item:22}  $\deg_{\bfu_{i}}(\Res_{\bfcA})>0$;
\item \label{item:23} $i\in J$.
     \end{enumerate}
   \end{enumerate}
\end{proposition}

\begin{proof}
  We first prove \eqref{item:16}. The equivalence between
  \eqref{item:12} and \eqref{item:19} follows directly from~\eqref{eq:77}.

By Proposition \ref{prop:1}, we have
  that $\Elim_{\bfcA}\neq1$ if and only if $\Elim_{\bfe_{0},\dots,
    \bfe_{n}}(|Z_{\bfcA}|)\ne 1$.  By \cite[Lemmas 1.34 and
  1.37(2)]{DKS2011:hvmsan}, this is equivalent to 
\begin{equation}\label{eq:68}
\dim(\pr_{I}(|Z_{\bfcA}|))\ge \#I - 1 \quad \text{ for all } I\subset\{0,\dots, n\},
\end{equation}
where $\pr_{I}$ denotes the projection $\prod_{i=0}^{n}\P^{c_{i}}
\rightarrow \prod_{i\in I}\P^{c_{i}}$. 
We claim that this condition is equivalent to \eqref{item:3}.

To prove this, suppose that \eqref{eq:68} holds. In particular,
$\dim(|Z_{\bfcA}|)=n$ and so $|Z_{\bfcA}|=X_{\bfcA}$.  Hence,
$\pr_{I}(|Z_{\bfcA}|)=\pr_{I}(X_{\bfcA})= X_{\bfcA_{I}}$. Applying
\eqref{eq:39}, we deduce that
$\dim(\pr_{I}(|Z_{\bfcA}|))=\rank(L_{\bfcA_{I}})$ and so
\eqref{item:3} follows. Conversely, suppose that \eqref{item:3}
holds. In particular, $ \rank(L_{\bfcA})=n$. By \eqref{eq:39}, this
implies that 
$\dim(X_{\bfcA})=n$ and so $|Z_{\bfcA}|=X_{\bfcA}$. Hence
$\dim(\pr_{I}(|Z_{\bfcA}|))=\dim(\pr_{I}(X_{\bfcA}))=\rank(L_{\bfcA_{I}})
\ge \#I-1$, and \eqref{eq:68} follows.

 
We now show the equivalence of \eqref{item:3} and the existence of a
unique essential subfamily of supports.  First, assume that
\eqref{item:3} holds.  Lemma \ref{lemm:6} applied to the subset
$I=\{0,\dots, n\}$ shows that there exists at least one essential
subfamily~$\bfcA_{J}$. Suppose that there exist a further essential
subfamily $\bfcA_{J'}$. Then
  \begin{displaymath}
    L_{\bfcA_{J\cup
        J'}}=L_{\bfcA_{J}}+L_{\bfcA_{J'}} \quad \text{ and }  \quad L_{\bfcA_{J\cap
        J'}}\subset L_{\bfcA_{J}}\cap L_{\bfcA_{J'}}.    
  \end{displaymath}
  We deduce that
  \begin{multline} \label{eq:4}
    \rank(L_{\bfcA_{J\cup J'}}) \le     \rank(L_{\bfcA_{J}}) + 
   \rank(L_{\bfcA_{J'}}) -\rank(L_{\bfcA_{J\cap J'}})  \\ 
\le \# J -1+\#J' - 1 -\#(J\cap J')
= \#(J\cup J')-2,
  \end{multline}
  since both $\bfcA_{J}$ and $\bfcA_{J'}$ are essential and
  $\bfcA_{J\cap J'}$ is a proper subfamily of them. The inequality
  \eqref{eq:4} contradicts \eqref{item:3}, showing that there is a
  unique essential subfamily.

  Conversely, suppose that \eqref{item:3} does not hold. Then, there
  exists a subset $I_{0}\subset\{0,\dots, n\}$ such that
\begin{math}
  \rank(L_{\bfcA_{I_{0}}})\le \#I -2.
\end{math}
By Lemma \ref{lemm:6}, there exists  $J\subset I_{0}$
such that $\bfcA_{J}$ is essential. Choose $i_{0}\in J$. Then
\begin{math}
  \rank(L_{\bfcA_{I_{0}\setminus \{i_{0}\}}})\le \#(I_{0}\setminus
  \{i_{0}\})-1.
\end{math}
Again, Lemma \ref{lemm:6} implies that there exists an essential
subfamily of support $\bfcA_{J'}$ with $J'\subset I_{0}\setminus
\{i_{0}\}$. By construction, the essential subfamilies  $\bfcA_{J}$
and $\bfcA_{J'}$ are different, concluding the proof of
\eqref{item:16}.

We now turn to the proof of \eqref{item:17}.  Suppose that
$\Elim_{\bfcA}\ne 1$ or $\Res_{\bfcA}\ne 1$ and let $\bfcA_{J}$ denote
the unique essential subfamily. The equivalence between
\eqref{item:21} and \eqref{item:22} follows again from \eqref{eq:77}. 

Choose $i\notin J$. Then $J\subset \{0,\dots, n\}\setminus \{i\}$ and
$\rank(L_{\bfcA_{J}})= \# J -1$. By \cite[Theorem
5.1.7]{Schneider:cbbmt}, we have that
  \begin{math}
    \MV_{M}(\Delta_{0}, \dots,
  \Delta_{i-1},\Delta_{i+1},\dots, \Delta_{n})=0.
  \end{math}

  Now let $i\in J$. There is no essential subfamily of supports
  $\bfcA_{J'}$ with $J'\not\ni i$. Lemma~\ref{lemm:6} then implies
  that $\rank(L_{\bfcA_{I}})\ge \# I $ for all $I\subset \{0,\dots,
  n\}\setminus \{i\}$. Applying again
  \cite[Theorem~5.1.7]{Schneider:cbbmt}, we deduce that
  $\MV_{M}(\Delta_{0}, \dots, \Delta_{i-1},\Delta_{i+1},\dots,
  \Delta_{n})>0$, as stated.
\end{proof}

Given a family of Laurent polynomials $f_{i}\in \C[M]$ with support
contained in $\cA_{i}$, $i=0,\dots, n$, we denote by
\begin{displaymath}
\Elim_{\bfcA}(f_{0},\dots, f_{r}),\quad \Res_{\bfcA}(f_{0},\dots,
f_{r}) \quad \in \C
\end{displaymath}
the evaluation of the sparse eliminant and the sparse resultant,
respectively, at the coefficients of the $f_{i}$'s.

Typically, the fact that the family of Laurent
polynomials has a common root in the torus implies the vanishing of
the sparse eliminant and of the sparse resultant. In precise terms, if
$\ov{\pi(\Omega_{\bfcA})}$ is a hypersurface,
\begin{equation}
  \label{eq:70}
V(f_{0},\dots,f_{n})\ne\emptyset \Longrightarrow \Res_{\bfcA}(f_{0},\dots, f_{r})=0,
\end{equation}
and a similar statement holds for the sparse eliminant.  In Lemma
\ref{lemm:7} below, we give sufficient conditions such that the
vanishing of the sparse eliminant at a given family of Laurent
polynomials implies the existence of a common root in the torus.

\begin{lemma}
\label{lemm:7} Let 
\begin{equation} \label{eq:3}
  \bff=(f_{0},\dots, f_{n})\in V(\Elim_{\bfcA})\setminus \bigcup_{i=0}^{n}\bigcup_{j=0}^{c_{i}}
  V\bigg(\frac{\partial\Elim_{\bfcA}}{\partial u_{i,j}}\bigg) \subset\P^{\bfc}.
\end{equation}
Then, $V(\bff) \ne \emptyset$ and,  for all $\xi\in V(\bff)$, 
\begin{equation} \label{eq:16}
  (\chi^{a_{i,j}}(\xi))_{0\le i\le n, 0\le j\le c_{i}}=
  \bigg(\frac{\partial\Elim_{\bfcA}}{\partial
    u_{i,j}}(\bff)\bigg)_{0\le i\le n, 0\le j\le c_{i}}\in \P^{\bfc}.
\end{equation}
\end{lemma}

\begin{proof} 
  If $\Elim_{\bfcA}=1$, then $V(\Elim_{\bfcA})=\emptyset$ and the
  statement is trivially verified. Hence, we suppose that
  $\Elim_{\bfcA}\ne 1$. 

  By Proposition \ref{prop:3}\eqref{item:16} and \eqref{eq:39},
  $\dim(X_{\bfcA})=n$. Hence, by the definition of the  toric cycle
  $Z_{\bfcA}$ in \eqref{eq:27}, it
  follows that $|Z_{\bfcA}|=X_{\bfcA}$. By Proposition \ref{prop:1},
  $\Elim_{\bfcA}=\Elim_{\bfe_{0},\dots, \bfe_{n}}({X_{\bfcA}})$ and so
  $ \ov{\pi(\Omega_{\bfcA})} ={\rho(\Omega_{X_{\bfcA}, (\bfe_{0},
      \dots, \bfe_{n})})}$.  In particular, the latter is a
  hypersurface that contains the point $\bff$.  By \eqref{eq:56},
  \begin{equation}\label{eq:69}
X_{\bfcA}\cap V(\ell_{0}, \dots, \ell_{n}) \ne \emptyset,
  \end{equation}
  where $\ell_{i}$ denotes the linear form on $\P^{c_{i}}$ associated
  to $f_{i}$ {\it via} the monomial map $\varphi_{\bfcA}$ given in
  \eqref{eq:23}.

  Take a point $\bfzeta\in X_{\bfcA}\cap V(\ell_{0}, \dots, \ell_{n})$
  and, for $j=0,\dots, n$, choose $0\le l_{j}\le c_{j}$ such that
  $\zeta_{j,l_{j}}\ne 0$. We assume without loss of generality that
  $\zeta_{j,l_{j}}=1$ for all $j$.  By \cite[Proposition
  1.37]{DKS2011:hvmsan}, there exists $\kappa\gg0$ such that 
  \begin{equation*}
    \Big(\prod_{j=0}^{n} u_{j,l_{j}}^{\kappa} \Big) \Elim_{\bfe_{0},\dots,
      \bfe_{n}}({X_{\bfcA}})\in (L_{0},\dots,L_{n})\quad \subset \C[\bfu][\bfx]/I(X_{\bfcA}),     
  \end{equation*}
  where $L_{i}$ denotes the general linear form as in
  \eqref{eq:71}. Choose $G_{j}\in \C[\bfu][\bfx]$ such that
  \begin{equation}\label{eq:25}
    \Big( \prod_{j=0}^{n}
    u_{j,l_{j}}^{\kappa} \Big) \Elim_{\bfe_{0},\dots, \bfe_{n}}({X_{\bfcA}})
    =\sum_{j=0}^{n}G_{j}L_{j}  \quad \pmod{I(X_{\bfcA})\otimes \C[\bfu]}.
  \end{equation}
  Computing partial derivatives, evaluating at the point $(\bfzeta,
  \bff)$ and using the fact that $\Elim_{\bfcA}=\Elim_{\bfe_{0},\dots,
    \bfe_{n}}({X_{\bfcA}})$, we deduce from \eqref{eq:25} that
  \begin{equation*}
  \frac{\partial\Elim_{\bfcA}}{\partial u_{i,j}}(\bff)= G_{i}(\bff,
  \bfzeta) \zeta_{i,j} \quad \text{ for } i=0,\dots, n  \text{ and }  j=0,\dots, c_{i}.
  \end{equation*}
  By the choice of $\bff$ in \eqref{eq:3}, 
  \begin{equation}
    \label{eq:72}
  ( \zeta_{i,j})_{i,j}=
  \bigg(\frac{\partial\Elim_{\bfcA}}{\partial
    u_{i,j}}(\bff)\bigg)_{i,j}\in \P^{\bfc}.
  \end{equation}
It follows that $\bfzeta\in X_{\bfcA}\setminus
  \bigcup_{i,j}V(x_{i,j})$. By Lemma \ref{lemm:1}, this latter subset
  coincides with the image of the map $\varphi_{\bfcA}$. It follows
  that $\varphi_{\bfcA}^{-1}(\bfzeta) $ is a nonempty subset of
  $V(\bff)$, proving the first statement.

  Now let $\xi\in V(\bff)$. The point $\bfzeta=\varphi_{\bfcA}(\xi)$
  satisfies \eqref{eq:69} and so it also satisfies \eqref{eq:72},
  which implies the formula~\eqref{eq:16} and completes the proof.
\end{proof}
 
\begin{proposition} \label{prop:9} Suppose that $L_{\bfcA}=M$ and that
  $\bfcA$ is essential. Then
  \begin{displaymath}
       \deg(\pi_{\bfcA}|_{\Omega_{\bfcA}})= 1  \quad \text{ and }
       \quad \Res_{\bfcA}=\pm\Elim_{\bfcA}.
  \end{displaymath}
\end{proposition}

\begin{proof}
 As $\bfcA$ is the unique essential subfamily, by
  Proposition \ref{prop:3}\eqref{item:16} we have that $\ov
  {\pi(\Omega_{\bfcA})}$ is a hypersurface of $\P^{\bfc}$ with
  defining equation $\Elim_{\bfcA}$. Consider the open subset of this
  hypersurface given by
\begin{displaymath}
U=V(\Elim_{\bfcA}) \setminus \bigcup_{i=0}^{n}\bigcup_{j=0}^{c_{i}}
  V\bigg(\frac{\partial\Elim_{\bfcA}}{\partial u_{i,j}}\bigg).
\end{displaymath}
By Proposition \ref{prop:3}\eqref{item:17}, $\deg_{\bfu_{i}}(\Elim_{\bfcA})>0$ for
all $i$ and so $U\ne\emptyset$.

Take $\bff\in U$. By Lemma \ref{lemm:7}, $V(\bff)\ne\emptyset$ and,
given $\xi\in V(\bff)\subset \T_{M}$, one can compute
$\chi^{a-b}(\xi)$ for all $a,b\in \cA_{i}$, $i=0,\dots, n$, in terms
of $\bff$. Hence, one can compute $\chi^{a}(\xi)$ for all $a\in
L_{\bfcA}$. Since $L_{\bfcA}=M$, it follows that $\xi$ is univocally
determined and so
\begin{displaymath}
  \# \pi_{\bfcA}(\bff) =1 \quad \text{ for all } \bff\in U.
\end{displaymath}
By \cite[\S II.6, Theorem 4]{Shafarevich:bag},  $
\deg(\pi_{\bfcA}|_{\Omega_{\bfcA}})= 1$, which proves the first
statement.

The second statement follows directly from the first one and \eqref{eq:77}.
\end{proof}

Suppose that $\Elim_{\bfcA}\ne 1$ and let $\bfcA_{J}$ be the unique
essential subfamily of supports.  For each $i\in J$, choose $b_{i}\in
M$ such that $\cA_{i}-b_{i}\subset L_{\bfcA_{J}}$.  Then 
$L_{\bfcA_{J}}$ has rank $\#J-1$ and $\cA_{i}-b_{i}$, $i\in J$, is a
family of nonempty finite subsets of $ L_{\bfcA_{J}}$. We define
$\Elim_{\bfcA_{J}}\in\Z[\{\bfu_{i}\}_{i\in J}]$ as the sparse
eliminant associated to the lattice $L_{\bfcA_{J}}$ and this family of
supports.  This polynomial does not depend on the choice of the
vectors $b_{i}$.

\begin{proposition}
  \label{prop:7} Suppose that
  $\Elim_{\bfcA}\ne 1$ and let  $\bfcA_{J}$ be the unique
  essential subfamily of $\bfcA$. Then
  \begin{displaymath}
 \Elim_{\bfcA}=\pm\Elim_{\bfcA_{J}}   
  \end{displaymath}
and, for $i\in J$, 
\begin{equation*}
  \deg_{\bfu_{i}}(\Elim_{\bfcA})=
  \MV_{L_{\bfcA_{J}}}(\{\Delta_{j}-b_{j}\}_{j\in J\setminus \{i\}}).
\end{equation*}
\end{proposition}

\begin{proof}
  The inclusion of lattices $L_{\bfcA_{J}}\hookrightarrow {M}$ induces
  a surjective homomorphism of tori $\psi\colon\T_{M}\to
  \T_{L_{\bfcA_{J}}}$.  Consider the incidence variety
  $\Omega_{\bfcA_{J}} \subset\T_{L_{\bfcA_{J}}}\times\prod_{i\in
    J}\P^{c_{i}}$.  Then there is a commutative diagram
\begin{displaymath}
  \xymatrix{\Omega_{\bfcA}\ar[r]^{\psi\times\pr_{J}}\ar[d]_{\pi} & \Omega_{\bfcA_{J}} \ar[d]^{\pi_{J}} \\
\pi(\Omega_{\bfcA})\ar[r]^{\pr_{J}} & \pi_{J}(\Omega_{\bfcA_{J}}) }
\end{displaymath}
where $\pi_{J}$ and $\pr_{J}$ are  induced by the projections
\begin{math}
  \T_{L_{\bfcA_{J}}}\times
\prod_{i\in{J}} \P^{c_{i}} \to \prod_{i\in{J}}
\P^{c_{i}}
\end{math}
and $\P^{\bfc}\to \prod_{i\in J}\P^{c_{i}}$, respectively.  Let $\Q[\{\bfu_{i}\}_{i\in J}]\hookrightarrow
\Q[\bfu]$ be the inclusion of algebras corresponding to the arrow in
the bottom row. Then there is an inclusion of ideals
\begin{displaymath}
  (\Elim_{\bfcA})\cap \Q[\{\bfu_{i}\}_{i\in J}]\supset (\Elim_{\bfcA_{J}}). 
\end{displaymath}
The hypothesis that $\Elim_{\bfcA}\ne 1$ implies that both ideals are
principal and irreducible. We conclude that
$\Elim_{\bfcA}=\pm\Elim_{\bfcA_{J}}$, which gives the first
statement. 

The second statement follows from the first one together with
Propositions \ref{prop:9} and \ref{prop:6}.
\end{proof}

\begin{lemma}
  \label{lemm:10}
  Let $L\subset M$ be a saturated sublattice of rank $m$ and $P_{i}$,
  $i=1,\dots, n$, convex bodies of $M_{\R}$ such that $P_{i}\subset L_{\R}$ for
  $i=1,\dots, m$. Then
  \begin{equation} \label{eq:9}
    \MV_{M}(P_{1}, \dots,
    P_{n})=
    \MV_{L}(P_{1}, \dots,
    P_{m})\MV_{M/L}(\varpi(P_{m+1}), \dots, \varpi(P_{n})),
  \end{equation}
  where $\varpi$ denotes the projection $M_{\R}\to M_{\R}/L_{\R}$.
\end{lemma}

\begin{proof} 
  The fact that $L$ is saturated implies that there is an isomorphism
  $M\simeq \Z^{n}$ identifying $L$ with $\Z^{m}\times
  \{\bfzero\}$. The mixed volumes in \eqref{eq:9} are invariant under
  isomorphism of lattices, and so it suffices to prove this formula in
  the case when $M=\Z^{n}$ and $L=\Z^{m}\times \{\bfzero\}$.

Let $P, Q\subset \R^{n}$ be compact bodies such that $P\subset
\R^{m}\times \{\bfzero\}$. The function on $\R_{\ge0}$ given by
$\lambda\mapsto \vol_{\Z^{n}}(\lambda P+Q)$
is polynomial in $\lambda,$ and
\begin{equation}
  \label{eq:73}
\MV_{\Z^{n}}(\overbrace{P,\dots, P}^{m},\overbrace{Q,\dots, Q}^{n-m})=
\coeff_{\lambda^{m}}(\vol_{\Z^{n}}(\lambda P+Q)).
\end{equation}
Let $\lambda\in \R_{\ge 0}$. By Fubini's theorem,
\begin{multline}
  \label{eq:74}
  \vol_{\Z^{n}}(\lambda P+Q)= \int_{\R^{n-m}} \vol_{\Z^{m}}((\lambda
  P+Q) \cap
  (\R^{m}+\bfx) )\dd \bfx\\=  \int_{\R^{n-m}} \vol_{\Z^{m}}(\lambda
  P+(Q \cap
  (\R^{m}+\bfx) ) \dd \bfx.
\end{multline}
with $\dd\bfx=\dd x_{1}\dots \dd x_{n-m}$. The $m$-dimensional volume of $\lambda P+(Q
\cap (\R^{m}+\bfx)) $ is different from 0 if and only if $Q \cap
(\R^{m}+\bfx)\ne \emptyset$ or, equivalently, if and only if $\bfx\in
\varpi(Q)$.  In that case, $\coeff_{\lambda^{m}}(\vol_{\Z^{m}}(\lambda
P+(Q \cap (\R^{m}+\bfx) ))= \vol_{\Z^{m}}(P)$. Hence,
\begin{multline}\label{eq:75}
  \coeff_{\lambda^{m}}\Big( \int_{\R^{n-m}} \vol_{\Z^{m}}(\lambda
  P+(Q \cap
  (\R^{m}+\bfx) ) \dd \bft\Big)\\ = \vol_{\Z^{m}}(P) \int_{\varpi(Q)}
  \dd\bfx=  \vol_{\Z^{m}}(P) \vol_{\Z^{n-m}}(\varpi(Q)).
\end{multline}
By \eqref{eq:73}, \eqref{eq:74} and \eqref{eq:75}, it follows that
\begin{equation*}
  \MV_{\Z^{n}}(\overbrace{P,\dots, P}^{m},\overbrace{Q,\dots,
    Q}^{n-m})= \vol_{\Z^{m}}(P) \vol_{\Z^{n-m}}(\varpi(Q)), 
\end{equation*}
which gives the formula \eqref{eq:9} for the case when
$P_{1}=\dots=P_{m}=P$ and $P_{m+1}=\dots=P_{n}=Q$. The general case
follows by a standard polarization argument.
\end{proof}

The following result shows that the degree of the restriction of
$\pi$ to the incidence variety $\Omega_{\bfcA}$ and, {\it a
  fortiori}, the relation between the sparse resultant and the sparse
eliminant, can be expressed in combinatorial terms. This formula
already appears in \cite[Theorem 2.23]{Esterov:dsnp}.

\begin{proposition} \label{prop:10} Suppose that $\Res_{\bfcA}\ne1$
  and let $\bfcA_{J}$ be the unique essential subfamily of
  $\bfcA$. Then
  \begin{displaymath}
    \deg(\pi_{\bfcA}|_{\Omega_{\bfcA}})= [L_{\bfcA_{J}}^{\sat}:L_{\bfcA_{J}}]
    \MV_{M/L_{\bfcA_{J}}^{\sat}}(\{\varpi(\Delta_{i}) \}_{i\notin J}), 
  \end{displaymath}
  where $L_{\bfcA_{J}}^{\sat}= (L_{\bfcA_{J}}^{\sat}\otimes\Q) \cap M$
  denotes the saturation of the sublattice $L_{\bfcA_{J}}$, and $\varpi$
  the projection $M_{\R}\to M/L_{\bfcA_{J}}^{\sat}\otimes \R$.  In
  particular,
\begin{displaymath}
  \Res_{\bfcA}= \pm\Elim_{\bfcA}^{[L_{\bfcA_{J}}^{\sat}:L_{\bfcA_{J}}]
    \MV_{M/L_{\bfcA_{J}}^{\sat}}(\{\varpi(\Delta_{i}) \}_{i\notin J})}.
\end{displaymath}
\end{proposition}

\begin{proof}
  Suppose for simplicity that $J=\{0,\dots, m\}$ and set $L=
  L_{\bfcA_{J}}$ for short. By comparing the degree with respect to
  $\bfu_{0}$ of $\Res_{\bfcA}$ and of $\Elim_{\bfcA}$ using
  Propositions \ref{prop:6} and \ref{prop:7}, we deduce that
\begin{displaymath}
    \deg(\pi_{\bfcA}|_{\Omega_{\bfcA}})=\frac{\MV_{M}(\Delta_{1},\dots,
      \Delta_{n})}{\MV_{L}(\Delta_{1}-a_{1,0},\dots, \Delta_{m}-a_{m,0})}.
\end{displaymath}

We have that $[L^{\sat}:L] \vol_{L}=\vol_{L^{\sat}}$ and so $
[L^{\sat}:L]\MV_{L} = \MV_{L^{\sat}}$. Lemma~\ref{lemm:10} then
implies that
\begin{displaymath}
    \deg(\pi_{\bfcA}|_{\Omega_{\bfcA}})=  [L^{\sat}:L]
    \MV_{M/L}^{\sat}(\varpi(\Delta_{m+1}),\dots,\varpi(\Delta_{n})),
\end{displaymath}
which proves the first statement. The second claim follows then
from \eqref{eq:77}.
\end{proof}

\begin{example} \label{exm:5} Let $\cA_{0},\dots, \cA_{n}$ be a family
  of $n+1$ nonempty finite subsets of $M$ with $\cA_{0}=\{a\}$ for
  $a\in M$. Suppose that $\cA_{0}$ is the unique essential subfamily,
 and set $ \Delta_{i}=\conv(\cA_{i})$, $i=1,\dots, n$. By
  Propositions \ref{prop:7} and \ref{prop:10}, it follows that
  \begin{displaymath}
    \Elim_{\bfcA}= \pm u_{0,a}, \quad \Res_{\cA}=
    \pm u_{0,a}^{\MV_{M}(\Delta_{1}, \dots,\Delta_{n})}
  \end{displaymath}
and $\deg(\pi_{\bfcA}|_{\Omega_{\bfcA}})={\MV_{M}(\Delta_{1},
  \dots,\Delta_{n})}$.
\end{example}

In \cite{Som04}, the second author gave a bound for the height of the
$\bfcA$-eliminant in the case when the family $\bfcA$ is
essential. The following result extends this bound to an arbitrary
family of supports.  Recall that, given a polynomial
$R=\sum_{\bfa}\alpha_{\bfa}\bfu^{\bfa}\in \Z[\bfu]$, its \emph{height}
and its \emph{sup-norm} are respectively defined as
\begin{displaymath}
  \h(R)=\log(\max_{\bfa}|\alpha_{\bfa}|)\quad \text{ and } \quad
  \|R\|_{\sup}=\sup_{|u_{i,j}|=1}|R(\bfu)| . 
\end{displaymath}

\begin{proposition} \label{prop:4} Let notation be as above. Then 
\begin{equation*}
\h(\Res_{\bfcA})\le \sum_{i=0}^{n}\MV_{M}(\Delta_{0}, \dots,
  \Delta_{i-1},\Delta_{i+1},\dots, \Delta_{n}) \log (\#\cA_{i}).
\end{equation*}
\end{proposition}

\begin{proof}
  We suppose that $\Res_{\bfcA}\ne1$ because otherwise, the inequality
  is trivially satisfied. Let $\bfcA_{J}$ be the unique essential
  subfamily of supports. By \cite[Lemma~1.3]{Som04},
\begin{displaymath}
  \log \|\Elim_{\bfcA_{J}}\|_{\sup}\le
  \frac{1}{[L_{\bfcA_{J}}^{\sat}:L_{\bfcA_{J}}]} \sum_{i\in
    J}\MV_{L_{\bfcA_{J}}^{\sat}}(\{\Delta_{j}-a_{j,0}\}_{j\ne i} )\log
  (\#\cA_{i}).
\end{displaymath}
Multiplying both sides of this
inequality by $\deg(\pi_{\bfcA}|_{\Omega_{\bfcA}})$ and applying
Proposition~\ref{prop:10}, it follows that
\begin{equation*}
  \log \|\Res_{\bfcA}\|_{\sup}\le  \sum_{i\in
    J}\MV_{M/L_{\bfcA_{J}}^{\sat}}(\{\varpi(\Delta_{k})\}_{k\notin
    J})\MV_{L_{\bfcA_{J}}^{\sat}}(\{\Delta_{j}-a_{j,0}\}_{j\in
    J\setminus \{i\}} )\log
  (\#\cA_{i}).
\end{equation*}
For short, write $\mu_{i}$ for the product of the two mixed volumes in the
right-hand side of this formula. By Lemma \ref{lemm:10} and Propositions~\ref{prop:6} and
\ref{prop:3}\eqref{item:17}, 
\begin{displaymath}
\MV_{M}(\Delta_{0}, \dots,
  \Delta_{i-1},\Delta_{i+1},\dots, \Delta_{n})=
  \begin{cases}
    \mu_{i}& \text{ if } i\in J,\\
0 & \text{ if } i\notin J.
  \end{cases}
\end{displaymath}
It follows that 
\begin{displaymath}
  \log \|\Res_{\bfcA}\|_{\sup}\le  \sum_{i=0}^{n}\MV_{M}(\Delta_{0}, \dots,
  \Delta_{i-1},\Delta_{i+1},\dots, \Delta_{n}) \log (\#\cA_{i}).
\end{displaymath}
The statement follows from the fact that $\h(\Res_{\bfcA})\le \log
\|\Res_{\bfcA}\|_{\sup}$, the latter being  a consequence of Cauchy's integral
formula, see  page 1255 in {\it loc. cit} for details.
\end{proof}

\section{The Poisson formula}\label{sec:poiss-form-form}

In this section, we prove the Poisson formula in Theorem \ref{thm:3}.
We also derive some of its consequences, including the formula for the
product of the roots in Corollary~\ref{cor:4}, the product formula for
the addition of supports, and the extension of the ``hidden variable''
technique to the sparse setting.

We keep the notation at \eqref{eq:98}. Furthermore, we set 
\begin{equation}\label{eq:6}
  \ov \bfcA=(\cA_{1},\dots, \cA_{n}), \quad \ov
  \Delta=\sum_{i=1}^{n}\Delta_{i} \quad \text{ and } \quad \ov
  \bfF=(F_{1},\dots, F_{n})
\end{equation}

Let $\cB\subset M$ be a nonempty finite subset and $f= \sum_{b\in
  \cB}\beta_{b}\chi^{ b}\in K[M]$ a Laurent polynomial over a field $K$ with support contained
in $\cB$. Given $v \in N_{\R}$, we set
\begin{displaymath}
  \cB_{ v }=\{  b\in \cB\mid \langle  b, v \rangle =h_{\cB}(v)\}\quad
  \text{ and } \quad f_{ v }= \sum_{ b\in
  \cB_{ v }}\beta_{ b}\chi^{ b},
\end{displaymath}
with $h_{\cB}$ the support function of $\cB$ as in \eqref{eq:18}.
We also set $\F=\Q(\bfu_{1},\dots, \bfu_{n})$.

\begin{definition} \label{def:5} Let $ v \in N\setminus
  \{0\}$ and $ v ^{\bot}\subset M_{\R}$ the orthogonal subspace. Then
  $M\cap v ^{\bot}$ is a lattice of rank $n-1$ and, for $i=1,\dots,n$,
  there exists $ b _{i, v }\in M$ such that $\cA_{i,v }- b_{i, v }
  \subset M\cap v ^{\bot}$. The \emph{sparse resultant of
  $\cA_{1},\dots, \cA_{n}$ in the direction of $ v $}, denoted by $\Res_{\cA_{1,v
    },\dots,\cA_{n, v }}$, is defined as the sparse resultant of the
  family $\cA_{i,v }- b_{i, v }$, $i=1,\dots, n$, considered as a
  family of nonempty finite subsets of  $ M\cap v ^{\bot}$.
  
  Let $F_{i}\in \F[M]$ be the general polynomial with support $\cA_{i}$
  as in \eqref{eq:24}, $i=1,\dots,n$.  For each $i$, write $F_{i,v
  }= \chi^{ b_{i, v }}G_{i, v }$ for the general Laurent polynomial
  $G_{i, v }\in \F[M\cap v ^{\bot}]$ with support $\cA_{i, v }- b_{i,
    v }$.  The expression
\begin{displaymath}
\Res_{\ov \bfcA_v}(\ov \bfF_{v})=  \Res_{\cA_{1,v },\dots,\cA_{n,v }}(F_{1,v },\dots,F_{n,v
  }) \in \Z[\bfu_{1},\dots, \bfu_{n}]
\end{displaymath}
is defined as the evaluation of this directional sparse resultant at
the coefficients of the $G_{i, v }$'s.  These constructions are
independent of the choice of the $ b_{i, v }$'s.
\end{definition}

By Proposition \ref{prop:3}\eqref{item:16}, we have that
$\Res_{\cA_{1,v },\dots,\cA_{n,v }}\ne 1$ only if $ v $ is an
inner normal to a face of $\ov \Delta$ of dimension $n-1$.  In
particular, the number of non-trivial directional sparse resultants of
the family $\ov \bfcA$ is finite.

We first prove the following  Poisson formula for the
general Laurent polynomials.

\begin{theorem} \label{thm:1} Let notation be as in
  \eqref{eq:6}. Then
 \begin{equation}\label{poiss}
  \Res_{\bfcA}(\bfF)= \pm
  \prod_{v}\Res_{\ov \bfcA_{v}}(\ov \bfF_{v})^{-h_{\cA_{0}}(v)}
  \cdot \prod_{\xi} F_{0}(\xi)^{m_{\xi}} ,
\end{equation}
the first product being over the primitive vectors $v\in N$ and the
second  over the roots $\xi\in \T_{M,\ov \F}$ of $F_{1},\dots,
F_{n}$, and where  $m_\xi$ denotes the multiplicity of $\xi$ as in
\eqref{eq:86}.
\end{theorem}

\begin{proof} First suppose that $\dim(\Delta)\le n-1$.  By
  Proposition \ref{prop:3}\eqref{item:16}, the sparse resultant in the
  left-hand side of \eqref{poiss} is 1. Since $\dim(\ov \Delta)\le
  \dim(\Delta)\le n-1$, the family $\ov \bfF$ has no roots and so the
  second product in the right-hand side is also 1.  When $\dim(\ov
  \Delta)=n-2$, Proposition \ref{prop:3}\eqref{item:16} also implies
  that all directional sparse resultants of $\ov \bfcA$ in the first
  product of \eqref{poiss} are equal to 1. When $\dim(\ov
  \Delta)=n-1$, there are two directional sparse resultants which
  might be nontrivial, corresponding to a primitive normal vector of
  $\ov \Delta$ and its opposite. Both directional sparse resultants
  coincide, but they appear with opposite exponents in the first
  product of \eqref{poiss}. In all these cases, the formula reduces to
  the equality
  \begin{math}
1=\pm1.  
  \end{math}

  From now on, we assume that $\dim(\Delta)=n$.  Let $Z_{\bfcA}$ be
  the multiprojective toric cycle in \eqref{eq:27}. This cycle is
  defined over $\Q$ and so it can be considered as a cycle of
  $\P_{\Q}^{\bfn}$. Let 
  \begin{math}
    Z_{\bfcA,\F}=Z_{\bfcA}\times \Spec(\F)
  \end{math}
  be the cycle on $\P^{\bfn}_{\F}$ induced by the base change
  $\Q\hookrightarrow \F$. Consider the linear forms
  $L_{i}=\sum_{j=0}^{c_{i}} u_{i,j} x_{i,j} \in \F [\bfx]$,
  $i=1,\dots, n$, and set $\div(L_{i})$ for the corresponding Cartier
  divisor on $\P_{\F}^{\bfc}$. These Cartier divisors intersect
  $Z_{\bfcA,\F}$ properly, and applying \cite[Propositions 1.28 and
  1.40 and Corollary 1.38]{DKS2011:hvmsan} we deduce that
\begin{equation}\label{eq:10}
  \Res_{\bfe_{0},\dots, \bfe_{n}}(Z_{\bfcA})= \lambda_{1}
  \Res_{\bfe_{0}}\Big(Z_{\bfcA,\F}\cdot \prod_{i=1}^{n}\div(L_{i})\Big)= \lambda_{2}
  \prod_{\xi} F_{0}(\xi)^{m_{\xi}} 
\end{equation}
with $\lambda_{i}\in \F^{\times}$, the product in the right-hand side
being as in \eqref{poiss}.

Suppose for the moment that $a_{0,0}=0$. Then, by evaluating
\eqref{eq:10} at $F_{0}=1,$ we obtain that $\lambda_{2}=
\Res_{\bfe_{0},\dots, \bfe_{n}}(Z_{\bfcA})(1,\ov \bfF)$.  By
\cite[Propositions 1.40]{DKS2011:hvmsan} and
Proposition~\ref{prop:11}, there exist $\nu_{i} \in \Q^{\times}$ such
that
\begin{multline}\label{eq:17}
 \Res_{\bfe_{0},\dots, \bfe_{n}}(Z_{\bfcA})(1,\ov \bfF)= \nu_{1} \Res_{\bfe_{1}, \dots, \bfe_{n}} (Z_{\bfcA}\cdot
  \div(x_{0,0})) \\= \nu_{2} \prod_{\Gamma}\Res_{\bfe_{1}, \dots,
    \bfe_{n}} (Z_{\bfcA, {\Gamma}})(\ov \bfF_{v({\Gamma})})^{-h_{\cA_{0}}(v({\Gamma}))},
\end{multline}
the product being over the facets $\Gamma$ of $\Delta,$ and where
$v({\Gamma})$ denotes the primitive inner normal vector of
$\Gamma$.  By Proposition \ref{prop:1}, $\Res_{\bfe_{0},\dots,
  \bfe_{n}}(Z_{\bfcA})= \Res_{\bfcA}$ and, for each facet $\Gamma$, 
  \begin{equation}
    \label{eq:7}
\Res_{\bfe_{1}, \dots,
    \bfe_{n}} (Z_{\bfcA, {\Gamma}})=  \Res_{\ov \bfcA_{v({\Gamma})}}.
  \end{equation}
By Proposition \ref{prop:3}\eqref{item:3}, $
  \Res_{\ov \bfcA_{v}}=1$ for
  every primitive vector $v\in N$ which is not an inner normal to a
  facet of $\Delta$.

From \eqref{eq:10}, \eqref{eq:17} and \eqref{eq:7}, it follows that
\begin{equation} \label{eq:43}
\Res_{\bfcA}  = \nu_{2}
  \prod_{v}\Res_{\ov \bfcA_{v}}(\bfF_{v})^{-h_{\cA_{0}}(v)} \cdot \prod_{\xi} F_{0}(\xi)^{m_{\xi}} 
\end{equation}
with $\nu_{2}\in \Q^{\times}$, the product being over the primitive
vectors $v\in N$. 

To prove that $\nu_{2}$ is actually equal to $\pm1$, we will show its
$p$-adic valuation is zero for every prime $p$ of $\Z$. To do this,
let $p$ be such a prime and consider the $p$-adic valuation $\ord_{p}$
on $\Q$. We extend this valuation to the field
$\F(\bfu_{0})=\Q(\bfu_{0},\bfu_{1},\dots, \bfu_{n})$ as in
\eqref{eq:103}, and we also denote it by $\ord_{p}$. By
Proposition~\ref{prop:8}, $\ord_{p}(\xi)=0$ for every root $\xi\in
\T_{M,\ov{\F(\bfu_{0})}}$ of $F_{1},\dots, F_{n}$. Hence
\begin{displaymath}
\ord_{p}\bigg(   \prod_{\xi} F_{0}(\xi)^{m_{\xi}}\bigg)=0.
\end{displaymath}
It follows that $\ord_{p}(\nu_{2})=0$. Since this holds for every $p$, we deduce that
$\nu_{2}=\pm1$, which proves the theorem for the case when $a_{0,0}=0$.

%

In particular, let $a\in M$ and set $\cA_{0}=\{0,a\},$ and
$-\cA_{0}=\{0,-a\}$. Note that $-\cA_{0}$ is the translate of
$\cA_{0}$ by the point $-a$. By Proposition \ref{prop:5},
\begin{displaymath}
  \Res_{\cA_{0},\ov\bfcA}(u_{0,0}+u_{0,1}\chi^{a},
  \ov \bfF)
=\pm\Res_{-\cA_{0},\ov\bfcA}(u_{0,0}\chi^{-a}+u_{0,1}, \ov \bfF) .
\end{displaymath}
Since both $\cA_{0}$ and $-\cA_{0}$ contain 0, we can apply the
previous case to both presentations of this sparse resultant to deduce that
\begin{multline*}
  \prod_{v}\Res_{\ov \bfcA_{v}}(\ov \bfF_{v})^{-\min(0,\langle a,v\rangle)}
  \cdot \prod_{\xi} (u_{0,0}+u_{0,1}\chi^{a}(\xi))^{m_{\xi}}\\=
\pm   \prod_{v}\Res_{\ov \bfcA_{v}}(\ov \bfF_{v})^{-\min(0,\langle -a,v\rangle)}
  \cdot \prod_{\xi} (u_{0,0}\chi^{-a}(\xi)+u_{0,1})^{m_{\xi}}.
\end{multline*}
Using that $\min(0,\langle a,v\rangle)-\min(0,-\langle a,v\rangle)=
\langle a,v\rangle$, we deduce from here that
\begin{equation}
  \label{eq:2}
\prod_{\xi} \chi^{a }(\xi)^{m_{\xi}}= \pm
  \prod_{v}\Res_{\ov \bfcA_{v}}(\ov \bfF_{v})^{\langle a,v\rangle}.
\end{equation}

Now we consider the general case when $a_{0,0}$ is an arbitrary
element of $M$. Applying Proposition~\ref{prop:5}, the formula for the
case when $a_{0,0}=0$, and \eqref{eq:2}, we get
\begin{align*}
\Res_{\cA_{0}, \ov \bfcA}(F_{0}, \ov \bfF)&=
\Res_{\cA_{0}-a_{0,0},\ov \bfcA}(\chi^{-a_{0,0}}F_{0},
 \ov \bfF)\\ &= 
\pm
  \prod_{v}\Res_{\ov \bfcA_{v}}(\ov \bfF_{v})^{\langle
    a_{0,0},v\rangle -h_{\cA_{0}}(v)}
  \cdot \prod_{\xi} (\chi^{a_{0,0}}(\xi)F_{0}(\xi))^{m_{\xi}}\\ &= 
\pm
  \prod_{v}\Res_{\ov \bfcA_{v}}(\ov \bfF_{v})^{-h_{\cA_{0}}(v)}
  \cdot \prod_{\xi} F_{0}(\xi)^{m_{\xi}},
\end{align*}
completing the proof. 
\end{proof}

\begin{remark}\label{rem:4}
  Let notation be as in Theorem \ref{thm:1}. By the structure theorem
  for Artin rings, there is a decomposition into local Artin rings
\begin{equation*}
\F[M]/(F_{1},\dots,F_{n})  =\bigoplus_{\xi}A_{\xi},
\end{equation*}
where the direct sum is over the roots $\xi$ of the family $F_{i}$,
$i=1,\dots, n$. Each local Artin ring $A_{\xi}$ is a $\F$-algebra of
dimension $m_{\xi}$. Hence
\begin{equation}\label{eq:78}
  \deg(Z(F_{1},\dots, F_{n}))=\sum_{\xi}m_{\xi},\quad \text{ and } \quad \prod_{\xi}
  F_{0}(\xi)^{m_{\xi}}= \norm_{S/\F}(F_{0}),
\end{equation}
with $S=\F[M]/(F_{1},\dots,F_{n})$, and where $ \norm_{S/\F}(F_{0})$
denotes the norm of $F_{0}$ as an element of this $\F$-algebra that
is, the determinant of the $\F$-linear endomorphism of $S$ defined by
the multiplication by $F_{0}$.
\end{remark}

We now study the genericity conditions allowing to specialize the
Poisson formula~\eqref{poiss}.

\begin{lemma} \label{lemm:4} Let $f_{i},g_{i}\in \C[M]$, $i=1,\dots,
  n$, such that $V(f_{1},\dots, f_{n})\subset \T_{M}$ has dimension
  0. Let $t$ be a variable and consider the ideal
  \begin{equation*}
 I= (f_{1}+tg_{1},\dots,
  f_{n}+tg_{n})\subset \C[t][M]. 
  \end{equation*}
  Then $t$ is not a zero divisor modulo $I$.
\end{lemma}

\begin{proof}
  Let $V(I)$ be the subvariety of $\T_{M}\times \A^{1}$ defined by
  $I$. This ideal is generated by $n$ elements and so, as a
  consequence of Krull's Hauptidealsatz, all irreducible components of
  $V(I)$ have dimension $\ge1$.

We have that $I+(t)= (f_{1},\dots, f_{n},t)$ and so $V(I)\cap V(t)$ is
0-dimensional. This implies that, if $W$ is an irreducible component
of $V(I)$ such that $W\cap V(t)\ne \emptyset$, then $\dim(W)=1$.  Hence,
there is an open subset $U\subset \T_{M}\times \A^{1}$ containing the
hyperplane $ V(t)$ where the family
$f_{i}+tg_{i}$, $i=1,\dots, n$, forms a complete intersection. In
particular,  $I$ has no embedded components supported on
$U$. We conclude that $t$ does not belong to any of the associated
prime ideals of $I$ and so it is not a zero divisor modulo $I$.
\end{proof}

  \begin{lemma} \label{lemm:5} Let $f_{i}\in \C[M]$ with
support contained in     $\cA_{i}$ and $F_{i}$ the general Laurent
    polynomial with support $\cA_{i}$ as in \eqref{eq:24}, $i=0,\dots,
    n$. Set $D=\MV_{M}(\Delta_{1},\dots, \Delta_{n})$ and consider the quotient
    algebras
  \begin{equation*}
R=\C[M]/(f_{1},\dots, f_{n}), \quad     S=\C\otimes \F(t)[M]/(f_{1}+tF_{1},\dots,
  f_{n}+tF_{n}).
  \end{equation*}
  Suppose that $\dim_{\C}(R)=D$ and let $g_{k}\in \C[M]$, $k=1,\dots,
  D$, giving a basis of $R$ over~$\C$. Then,
\begin{enumerate}
\item \label{item:1} $\dim_{\C\otimes \F(t)}(S)= D$ and $g_{k}$, $k=1,\dots, D$,
  is a basis of $S$ over $\F(t)$;
\item \label{item:2} 
  \begin{math}
\norm_{S/\C\otimes \F(t)} (f_{0}+tF_{0})\big|_{t=0} =\norm_{R/\C}(f_{0}).
  \end{math}
\end{enumerate}
\end{lemma}

\begin{proof}
  We first prove \eqref{item:1}. Set $\L=\C\otimes \F$ for short. The family $f_{i}+tF_{i}$,
  $i=1,\dots, n$, verifies the hypothesis  of
  Bernstein's theorem in~\eqref{eq:46}. Then $V(f_{1}+tF_{1}, \dots,
  f_{n}+tF_{n})$ is of dimension 0 and, by~\eqref{eq:78}, 
\begin{displaymath}
  \dim_{\L(t)}(S)=\deg(Z(f_{1}+tF_{1}, \dots, f_{n}+tF_{n}))=D=\dim_{\C}(R). 
\end{displaymath}
Hence, to prove that the $g_{k}$'s form a basis of $S$ over $\L(t)$, it
is enough to show that they are linearly independent. Suppose that
this is not the case and take a nontrivial linear combination
\begin{equation} \label{eq:5}
  \sum_{l=1}^{D} \gamma_{l}g_{l} =0  \quad \text{ on } S
\end{equation}
with $\gamma_{l}\in \L(t)$, not all of them simultaneously zero.  Set
$I\subset \L[t][M]$ for the ideal generated in this ring by the family
$f_{i}+tF_{i}$, $i=1,\dots, n$.  Multiplying \eqref{eq:5} by a
suitable denominator in $\L[t]\setminus \{0\}$, we can assume without
loss of generality that $\gamma_{l}\in \C[\bfu_{1},\dots,
\bfu_{n}][t]$ and that $ \sum_{l=1}^{D} \gamma_{l}g_{l}\in I$.
Moreover, by Lemma \ref{lemm:4}, the variable $t$ is not a zero
divisor modulo $I$ and so we can also assume that $t\nmid
\gcd_{\L[t]}(\gamma_{1},\dots, \gamma_{D})$.

We then obtain a nontrivial linear combination over $\C$ for the
$g_{k}$'s by specializing~\eqref{eq:5} at $t=0$ and taking any nonzero
coefficient in the expansion with respect to the variables
$\bfu_{i}$. This contradicts our assumption and hence it follows that the
$g_{k}$'s form a basis of $S$ over~$\L(t)$, which proves  \eqref{item:1}.

Now we turn to \eqref{item:2}. For $j=0,\dots, c_{0}$ and  $k=1,\dots, D$
write
  \begin{equation}\label{eq:47}
\chi^{a_{0,j}}g_{k}=\sum_{l=1}^{D}p_{j,k,l}\, g_{l} \in R\quad \text{ and
}\quad    \chi^{a_{0,j}}g_{k}= \sum_{l=1}^{D}P_{j,k,l} \, g_{l} \in  S
  \end{equation}
  with $p_{j,k,l}\in \C$ and $P_{j,k,l}\in \L(t)$.  
Using the  fact that $t$
is not a zero divisor modulo $I$, we can deduce that none of the $P_{j,k,l}$'s has a
pole at $t=0$ and that $\chi^{a_{0,j}}g_{k}- \sum_{l=1}^{D}P_{j,k,l}
\, g_{l} \in I$. Evaluating the right equation in \eqref{eq:47} at
$t=0$ we obtain
\begin{equation}
  \label{eq:92}
\chi^{a_{0,j}}g_{k}= \sum_{l=1}^{D}P_{j,k,l}\Big|_{t=0} \, g_{l} \in
(f_{1},\dots, f_{n})\subset \L[M].
\end{equation}
Since the $g_{k}$'s are a basis of $R$ over $\C$, they are also a
basis of $R\otimes \L$ over $\L$. It then follows from \eqref{eq:47}
and \eqref{eq:92} that $P_{j,k,l}\big|_{t=0}=p_{j,k,l}\in \C$ for all
$j,k,l$. 

  Let $m_{f_{0}}$ and $m_{f_{0}}$ respectively denote the
  matrix of the multiplication by $F_{0}$ on $S$ and by $f_{0}$ on
  $R$, with respect to the basis $g_{k}$, $k=1,\dots,
  D$. 
Then, $$m_{f_{0}+tF_{0}}\big|_{t=0}= (m_{f_{0}}+t
m_{F_{0}})\big|_{t=0}=m_{f_{0}},$$ and hence 
  \begin{displaymath}
    \norm_{S/\L(t)}(F_{0})\big|_{t=0}=\det(m_{f_{0}+tF_{0}}\big| _{t=0})=
    \det(m_{f_{0}})=\norm_{R/\C}(f_{0}),
  \end{displaymath}
as stated.
\end{proof}

We finally prove the results stated in the introduction. 

\begin{proof}[Proof of Theorem \ref{thm:3} and Corollary \ref{cor:4}]
  By \eqref{eq:70}, the hypothesis $\Res_{\ov \bfcA_{v }}(\ov f_{v
  })\ne 0$ implies that the family $f_{i,v }$, $i=1,\dots, n$, has
  no roots in $\T_{M}$. Then, by Bernstein's theorem in~\eqref{eq:46},
  the variety $V(f_{1}, \dots, f_{n})$ is of dimension 0 and
\begin{displaymath}
  \dim_{\C}(\C[M]/(f_{1},\dots,
f_{n}))=\deg(Z(f_{1},\dots, f_{n}))=D.
\end{displaymath}
Then we can apply Lemma \ref{lemm:5}\eqref{item:2} and Remark
\ref{rem:4} to deduce that
\begin{multline}\label{eq:41}
 \prod_{\xi} f_{0}(\xi)^{m_{\xi}}= \norm_{R/\C}({f_{0}})  \\=
 \norm_{S/\C\otimes \F(t)}({f_{0}+tF_{0}})\big|_{t=0}= \bigg( \prod_{\xi}
(f_{0}(\xi)+t F_{0}(\xi))^{m_{\xi}}\bigg) \Big|_{t=0}.
\end{multline}
Applying the Poisson formula \eqref{poiss} to the general Laurent
polynomials $f_{i}+tF_{i}$, $i=0,\dots, n$, we deduce that the second
product in \eqref{eq:41} is equal to
\begin{equation}
  \label{eq:93}
\pm  {\Res_{\bfcA}(\bff+t\bfF)}\cdot{\prod_{v}\Res_{\ov \bfcA_{v}}(\ov
  \bff_{v}+t\ov \bfF_{v})^{h_{\cA_{0}}(v)}},
\end{equation}
the product being over the primitive vectors $v\in N$.  Theorem
\ref{thm:3} then follows from~\eqref{eq:41} by evaluation
\eqref{eq:93} at $t=0$.

Corollary \ref{cor:4} follows from Theorem \ref{thm:3} applied to the
 supports $\{a\}, \cA_{1},\dots, \cA_{n}$.
\end{proof}

From the Poisson formula, we can deduce a number of other properties for
the sparse resultant.  The following is the product formula for the
addition of supports.

\begin{corollary}\label{cor:3}
  Let $\cA_{0},\cA_{0}', \cA_{1}, \dots, \cA_{n}\subset M$ be nonempty finite
  subsets and $F_0, F'_0$, $F_{1},\dots, F_{n}$  the
  general Laurent polynomials with support $\cA_{0},\cA_{0}', \cA_{1},
  \dots, \cA_{n}$, respectively. Then
\begin{multline*}
\Res_{\cA_{0}+\cA_{0}', \cA_{1},\dots, \cA_{n}}(F_0 F_{0}', F_{1},\ldots,F_n)\\=
\pm\Res_{\cA_{0}, \cA_{1},\dots, \cA_{n}}(F_0, F_{1},\ldots,F_n)\cdot
\Res_{\cA_{0}', \cA_{1},\dots, \cA_{n}}(F_{0}', F_{1},\ldots,F_n).
  \end{multline*}
\end{corollary}

\begin{proof}
  This follows from Theorem \ref{thm:1} and the additivity of support
  functions with respect to the addition of sets.
\end{proof}





We devote the rest of this section to the proof of Theorem \ref{thm:4}
in the introduction.  Let $n\ge 1$ and set $M=\Z^{n}$ and let be the
general Laurent polynomials $F_{i}\in \Q[\bfu_{i}][t_{1}^{\pm1},
\dots, t_{n}^{\pm1}]$ with support $\cA_{i}$, $i=1,\dots, n$. Let $
\Res_{\cA_{1}, \dots, \cA_{n}}^{t_{n}}$ as defined in \eqref{eq:42}.

\begin{proposition} \label{prop:12} Let notation be as above. Then, there
  exists $d\in\Z$ such that
  \begin{equation}
    \label{eq:80}
\Res_{\cA_{1}, \dots, \cA_{n}}^{t_{n}}= \pm t_{n}^d \, \Res_{\{\bfzero,
  \bfe_{1}\}, \cA_{1}, \dots, \cA_{n}}(z-t_{n},F_1,\ldots,F_n)
\big|_{z=t_{n}},
  \end{equation}
with  $\bfe_{n}=(0,\dots, 0, 1)\in \Z^{n}$.
\end{proposition}

\begin{proof} 
  Let $\ov \bfcA=(\cA_{1},\dots, \cA_{n})$, $\ov \bfF=(F_{1},\dots,
  F_{n})$ and $\ov \bfu=\{\bfu_{1},\dots,
  \bfu_{n}\}$ as before, and set for short
  \begin{displaymath}
  R=\Res_{\ov\bfcA}^{t_{n}}\in \Q[\ov
  \bfu][t_n^{\pm1}] \quad \text{ and } \quad E= \Res_{\{\bfzero,
  \bfe_{1}\}, \ov \bfcA}(z-t_{n},\ov \bfF)
\in \Q[\ov \bfu][z].  
  \end{displaymath}
  Set also $\varpi \colon \R^{n}\rightarrow \R^{n-1}$ for the
  projection onto the first $n-1$ coordinates of $\R^{n}$.

  We will prove the statement by induction on the number of
  variables. When $n=1$, 
  \begin{equation}\label{eq:104}
R=\pm F_{1}\quad \text{ and } \quad E=z^{-\ord_{t_{1}}(F_{1})}F_{1}(z).
  \end{equation}
  These identities can be respectively proven using Example
  \ref{exm:5} and the formula \eqref{poiss}. This implies
  \eqref{eq:80} in this case, with $d=\ord_{t_{1}}(F_{1})$.

  Suppose now that $n\ge2$. Applying the  formula \eqref{poiss}
  both to $R$ and to $E$, we get
  \begin{align*}
    R&= \pm 
  \prod_{v  }\Res_{\varpi(\cA_{1})_{v},\dots,\varpi(\cA_{n-1})_{v  }}(F_{1,v  },\dots,F_{n-1,v  })^{-h_{\varpi(\cA_{n})}(v  )}
   \prod_{\xi} F_{n}(\xi)^{m_{\xi}} ,
\\ 
E(t_{n})&= \pm 
  \prod_{w }\Res_{\{\bfzero, \bfe_{n}\}_{w }, \cA_{1,w },\dots,\cA_{n-1,w }}((z-t_{n})_{w },F_{1,w },\dots,F_{n-1,w })^{-h_{\cA_{n}}(w )}
  \prod_{\eta} F_{n}(\eta)^{m_{\eta}}.
  \end{align*}
  In these formulae, the first product is over all primitive vectors
  $v  $ in $\Z^{n-1}$, the second is over the roots $\xi$ of
  $F_{1},\dots, F_{n-1}$ in $(\ov {\C(\bfu_{1},\dots,
    \bfu_{n-1})(t_{n})}^{\times})^{n-1}$, the third is over all
  primitive vectors $w $ in $\Z^{n}$, and the fourth is over the
  roots $\eta$ of $z-t_{n},F_{1},\dots, F_{n-1}$ in~$(\ov
  {\C(\bfu_{1},\dots, \bfu_{n-1})(z)}^{\times})^{n}$.

Using Remark \ref{rem:4}, we can verify that 
\begin{displaymath}
  \prod_{\xi} F_{n}(\xi)^{m_{\xi}}= \prod_{\eta}
  F_{n}(\eta)^{m_{\eta}} \bigg|_{z=t_{n}}
 \end{displaymath}
Let $w =(w_{1},\dots, w_{n})\in \Z^{n}$.  If $w $ is of the form $(v ,0 )$ with $v  \in \Z^{n-1}$, then
 ${h_{\cA_{n}}(w )}=h_{\varpi(\cA_{n})}(v  )$.  Applying the inductive
 hypothesis, we get that, in this case,
\begin{multline}\label{eq:101}
 \Res_{\varpi(\cA_{1})_{v  },\dots,\varpi(\cA_{n-1})_{v  }}(F_{1,v  },\dots,F_{n-1,v  })^{-h_{\varpi(\cA_{n})}(v  )}\\=  
t_{n}^{-h_{\cA_{n}}(w )d_{w}}\Res_{\{\bfzero, \bfe_{n}\},
  \cA_{1,w },\dots,\cA_{n-1,w }}(z-t_{n},F_{1,w },\dots,F_{n-1,w })^{-h_{\cA_{n}}(w )}\Big|_{z=t_{n}}
\end{multline}
with $d_{w}\in \Z$. On the other hand, if $w_{n}\ne0$, then 
\begin{displaymath}
  (z-t_{n})_{w }=
  \begin{cases}
    z& \text{ if } w_{n}>0,\\
-t_{n}& \text{ if } w_{n}<0.
  \end{cases}
\end{displaymath}
Example \ref{exm:5} implies that
\begin{equation} \label{eq:100}
  \Res_{\{\bfzero, \bfe_{n}\}_{w },
    \cA_{1,w },\dots,\cA_{n-1,w }}((z-t_{n})_{w },F_{1,w },\dots,F_{n-1,w })^{-h_{\cA_{n}}(w )}=\pm z^{c_{w}}
\end{equation}
with
\begin{displaymath}
c_{w}=
  \begin{cases}
    -h_{\cA_{n}}(w)\MV_{\Z^{n}\cap w^{\bot}}(\Delta_{1,w}, \dots, \Delta_{n,w})& \text{ if } w_{n}>0,\\
0& \text{ if } w_{n}<0,
  \end{cases}
\end{displaymath}
where $\Delta_{i,w}$ is the face in the direction $w$ of the convex
hull of $\cA_{i}$. The statement then follows from \eqref{eq:101} and
\eqref{eq:100} with
\begin{equation}
  \label{eq:105}
  d=-\sum_{w} h_{\cA_{n}}(w )d_{w} -
\sum_{w} h_{\cA_{n}}(w)\MV_{\Z^{n}\cap w^{\bot}}(\Delta_{1,w},
  \dots, \Delta_{n,w}) \in \Z,
\end{equation}
for $d_{w}$ as in \eqref{eq:101}.
\end{proof}

\begin{remark} \label{rem:2} The exponent $d$ in \eqref{eq:80} can be
  made explicit in terms of mixed integrals in the sense of
  \cite[Definition 1.1]{PS:rbke} or, equivalently, shadow mixed
  volumes as in \cite[Definition 1.7]{Esterov:emfb}. Indeed, let
  $\iota\colon \R^{n}\to \R^{n}$ given by
  $(x_{1},\dots,x_{n-1},x_{n})\mapsto
  (x_{1},\dots,x_{n-1},-x_{n})$. Then $d$ coincides with the mixed
  integral of the family of concave functions on
  $\varpi(\Delta_{i})\to \R$, $i=1,\dots,n$, parametrizing the upper envelope of
  $\iota(\Delta_{i})$. This can be shown by induction on the number of
  variables $n$ by using \eqref{eq:104}, plus the recursive formulae
  \eqref{eq:105} and \cite[(8.6)]{PS:rbke}.
\end{remark}

\begin{proof}[Proof of Theorem \ref{thm:4}] 
This follows directly from Proposition \ref{prop:12} and Theorem \ref{thm:3}.  
\end{proof}

\section{Comparison with previous
  results and further examples} \label{sec:example-comp-with}

Using the relation between sparse resultants and sparse eliminants
given in
Proposition~\ref{prop:10}, we can easily translate any  results for sparse
resultants in terms of sparse eliminants and viceversa: with 
notation
as in  Proposition~\ref{prop:10}, we have that 
\begin{displaymath}
  \Res_{\bfcA}= \pm \Elim_{\bfcA}^{d_{\bfcA}}
\end{displaymath}
with 
\begin{equation*} 
d_{\bfcA}=
\begin{cases}
 [L_{\bfcA_{J}}^{\sat}:L_{\bfcA_{J}}]
    \MV_{M/L_{\bfcA_{J}}^{\sat}}(\{\varpi(\Delta_{i}) \}_{i\notin J})
    & \text{ if } \exists ! \text{   essential subfamily } \bfcA_{J}, \\
0 & \text{ otherwise.}
\end{cases}
\end{equation*}

In particular, the Poisson formula in Theorem \ref{thm:1} can be
translated in terms of sparse eliminants as follows. Let  notation be
as in that result. For each primitive vector $v \in N$ we choose $b_{i,v}\in
M$ such that $\cA_{i,v }-b_{i,v }\subset M\cap v ^{\bot}\simeq \Z^{n-1}$,
$i=1,\dots, n$, and we set
\begin{equation*}
d_{\ov \bfcA_{v}}:=  d_{\cA_{1,v}-b_{1,v},\dots,\cA_{n,v}-b_{n,v}}.
\end{equation*}
Then, the formula \eqref{poiss} can be rewritten as
 \begin{equation}\label{eq:19}
  \Elim_{\bfcA}^{d_{\bfcA}}(\bfF)= \pm
  \prod_{v}\Elim_{\ov\bfcA_{v}}(\ov \bfF_{v})^{-d_{\ov\bfcA_{v}}h_{\cA_{0}}(v)}
  \cdot \prod_{\xi} F_{0}(\xi)^{m_{\xi}} .
\end{equation}

On the other hand, the product formula in \cite[Theorem 1.1]{PS93} can
be reformulated with our notation as
\begin{equation}
  \label{eq:82}
    \Elim_{\bfcA}= \lambda\cdot
  \prod_{v}\Elim_{\ov \bfcA_{v}}(\ov \bfF_{v})^{-\delta_{v}}
  \cdot \prod_{\xi} F_{0}(\xi)^{m_{\xi}} ,
\end{equation}
with $\lambda\in \Q^{\times}$ and where, for each primitive vector
$v\in N$, the exponent $\delta_{v}$ is given by
\begin{displaymath}
  \delta_{v}=
  \begin{cases}
[L^{\sat}_{\ov
  \bfcA_{v}}:L_{\ov \bfcA_{v}}] & \text{ if } v \text{ is normal to a facet of }
\ov \Delta, \\
 0 & \text{ otherwise.}     
  \end{cases}
\end{displaymath}
In \cite[Theorem 1.1]{PS93}, it is implicitly assumed that
$L_{\bfcA}=\Z^{n}$ and that the family $\bfcA$ is essential. These
assumptions imply that $d_{\bfcA}=1$. Hence, \eqref{eq:82} actually
holds if and only if, for every primitive vector $v\in N$ such that
$\Elim_{\ov \bfcA_{v}}\ne 1$,
\begin{equation*}
  \delta_{v}= d_{\ov\bfcA_{v}}h_{\cA_{0}}(v) .
\end{equation*}
This set of equalities does hold when, for each $v$ such that $\ov
\bfcA_{v}$ has a unique essential subfamily, this subfamily actually
coincides with $\ov \bfcA_{v}$. The Pedersen-Sturmfels product formula is
correct in that case, which includes the unmixed case
when $\cA_{0}=\dots=\cA_{n}=\cA$ for a nonempty finite subset
$\cA\subset \Z^{n}$ such that $L_{\cA}=\Z^{n}$.

Example \ref{exm:4} in the introduction illustrates how \eqref{eq:82}
can fail in degenerate cases.  In the setting of this example,
$L_{\bfcA}=\Z^{2}$ and $\bfcA$ is essential. However, for the vector
$(1,0)$, the unique essential subfamily $\ov\bfcA_{(1,0)}$ is the
point $\{(-1,0)\}$.  The exponent of the directional eliminant
$\Elim_{\cA_{1,(1,0)},\cA_{2,(1,0)}}=u_{1,1}$ in the formula
\eqref{eq:19} is  the 1-dimensional volume of the segment
$\conv((-1,0), (-1,2))$, which is equal to $2.$ On the other hand,
$\delta_{(1,0)}=1$ because $L_{\ov\bfcA_{(1,0)}}$ is saturated, and so 
\eqref{eq:82} fails in this case.

In \cite{Min03}, Minimair reformulated \eqref{eq:82} in the course of
his study of sparse resultants under vanishing coefficients, but this
reformulation has also flaws. In particular, the definition of the
exponent $\e_{\cA_{1},\dots, \cA_{n}}$ in \cite[Remark~3]{Min03}
depends on the construction of a supplement of the sublattice
$L_{\bfcA_{J}}$ associated to an essential subfamily of supports: if
this sublattice is not saturated, the supplement does not exists and
the exponent cannot be defined.  Moreover, \cite[Theorem 8]{Min03} is
meaningless in many situations as it leads to expressions of the form
$\frac00$ like the one shown in Example \ref{exm:4}.

Next we give two further examples. The first one shows that the
condition that $\bfcA$ is essential, which is implicitly assumed in
\cite{PS93}, is necessary for  \eqref{eq:82} to hold.

\begin{example} \label{exm:1} Let $M=\Z,$ and set $\cA_{0}=\{0\}$ and
  $\cA_{1}=\{0,1,2\}$.  Then $\cA_{0}$ is the unique essential
  subfamily and $\Res_{\bfcA}= \pm u_{0,0}^{2}$. We also have that
  $h_{\cA_{0}}(v)= 0 $ for all $v\in N$. Hence, the Poisson formula
  \eqref{poiss} reads in this case as
\begin{displaymath}
  \pm
u_{0,0}^{2} = \pm F_{0}(\xi_{1}) F_{0}(\xi_{2}), 
\end{displaymath}
where $\xi_{i}$ are the roots of $F_{1}$. On the other hand,
$\Elim_{\bfcA}= \pm u_{0,0}$ and so \eqref{eq:82}
does not hold.
    \end{example}

    The next example exhibits  a phenomenon similar to the one in Example
    \ref{exm:4}.

    \begin{example} \label{exm:2}
Let $M=\Z^{2}$ and set
\begin{displaymath}
\cA_{0}=\{(0,1),(1,0)\}, \quad       \cA_{1}=\{(0,0),(1,0)\}, \quad       
\cA_{2}=\{(0,0),(0,1), (0,2)\}. 
\end{displaymath}
Then $L_{\bfcA}=\Z^{2}$ and $\bfcA=(\cA_{0},\cA_{1}, \cA_{2})$ is
essential. It can be verified that
\begin{displaymath}
  \Res_{\bfcA}= u_{0,0}^{2}u_{1,0}^{2}u_{2,0} +
  u_{0,0}u_{0,1}
  u_{1,0}u_{1,1}u_{2,1} + u_{0,1}^{2}u_{1,1}^{2}u_{2,2},
\end{displaymath}
We can verify that the formula \eqref{poiss} reads in this case as
\begin{displaymath}
  \Res_{\bfcA}= \pm u_{1,1}^{2}u_{2,2}F_{0}(\xi_{1}) F_{0}(\xi_{2})
\end{displaymath}
where $\xi_{i}$ are the roots of the family $F_{1},F_{2}$.  We have
that $\Elim_{\bfcA}=\Res_{\bfcA}$ but the formula \eqref{eq:82} gives the exponent
1 to the directional sparse eliminant $u_{1,1}$. Hence, this formula also fails in this case.
    \end{example}

    The product formula for the addition of supports in Corollary
    \ref{cor:3} can also be rewritten in terms of sparse
    eliminants. Indeed, with notation as in that statement, set
    $\bfcA=(\cA_{0},\cA_{1},\cdots, \cA_{n})$,
    $\bfcA'=(\cA_{0}',\cA_{1},\cdots, \cA_{n})$,
    $\bfF=(F_{0},F_{1},\dots, F_{n})$ and $\bfF'=(F_{0}',F_{1},\dots,
    F_{n})$ for short. Then
  \begin{equation*}
\Elim_{\cA_{0}+\cA_{0}', \cA_{1},\dots, \cA_{n}}(F_0 F_{0}', F_{1},\ldots,F_n)^{d_{{\cA_{0}+\cA_{0}', \cA_{1},\dots, \cA_{n}}}}=
\pm\Elim_{\bfcA}( \bfF)^{d_{{\bfcA}}}\cdot \Elim_{\bfcA'}( \bfF')^{d_{{\bfcA'}}}.
  \end{equation*}
On the other hand, the analogous formula in \cite[Proposition
7.1]{PS93} can be reformulated with our notation as 
\begin{multline} \label{eq:81}
\Elim_{\cA_{0}+\cA_{0}', \cA_{1},\dots, \cA_{n}}(F_0 F_{0}',
F_{1},\ldots,F_n) \\ = \lambda\, \Elim_{\bfcA}( \bfF)^{[L_{{\bfcA}}:L_{\bfcA'}]}\cdot \Elim_{\bfcA'}( \bfF')^{[L_{{\bfcA}}:L_{\bfcA''}]}
\end{multline}
with $\lambda\in \Q^{\times}$. These two formulae are equivalent, up
to the scalar factor $\lambda$, in
the case when both $\bfcA'$ and $\bfcA''$ are essential. Otherwise,
\eqref{eq:81} might fail, as shown by the following example.

\begin{example} \label{exm:3} Let $M=\Z$ and set $\cA_{0}'=\{0\}$,
    $\cA_{0}''=\{0,1\}$ and $\cA_{1}=\{0,1,2\}$.  Then the formula in Corollary
    \ref{cor:3} reads in this case as
\begin{displaymath}
  \Elim_{\{0,1\},\{0,1,2\}}(u_{0,0}'(u_{0,0}''+u_{0,1}''x), f_{1}) = \pm
  {u}_{0,0}'^{\ 2} \, \Elim_{\{0,1\},\{0,1,2\}}(u_{0,0}''+u_{0,1}''x,f_{1}),
\end{displaymath}
since $\Res_{\{0\},\{0,1,2\}}={u}_{0,0}'^{\ 2}$.  However,
the formula \eqref{eq:81} gives the exponent 1 to the
sparse eliminant $\Elim_{\{0\},\{0,1,2\}}=u_{0,0}'$, instead of 2.
\end{example}

\newcommand{\noopsort}[1]{} \newcommand{\printfirst}[2]{#1}
\newcommand{\singleletter}[1]{#1} \newcommand{\switchargs}[2]{#2#1}
\def\cprime{$'$} \providecommand{\bysame}{\leavevmode\hbox
to3em{\hrulefill}\thinspace}
\providecommand{\MR}{\relax\ifhmode\unskip\space\fi MR } 
\providecommand{\MRhref}[2]{%
\href{http://www.ams.org/mathscinet-getitem?mr=#1}{#2} }
\providecommand{\href}[2]{#2}

\end{document}